\documentclass{article}
\usepackage[utf8]{inputenc}
\usepackage[T1]{fontenc}
\usepackage{comment}
\usepackage{mathtools}
\usepackage{tikz-cd} 
\usepackage{amsmath,amsfonts,amssymb,amsmath,latexsym,mathrsfs,amsthm}
\usepackage{stackengine}
\usepackage{afterpage}

\newcommand{\Q}{ {\mathbb Q} }
\newcommand{\R}{ {\mathbb R} }
\newcommand{\C}{ {\mathbb C} }
\newcommand{\Z}{ {\mathbb Z} }
\newcommand{\N}{ {\mathbb N} }

\newcommand{\GL}{\mathrm{GL}}
\newcommand{\SU}{\mathrm{SU}}
\newcommand{\PSU}{\mathrm{PSU}}
\newcommand{\PSL}{\mathrm{PSL}}
\newcommand{\PGL}{\mathrm{PGL}}
\newcommand{\SL}{\mathrm{SL}}
\newcommand{\PP}{ {\mathbb P} }
\newcommand{\Pe}{{\mathbb P}^1_* }
\newcommand{\Pec}{ Z}
\newcommand{\Cng}{  C_n^G(\PP^1\setminus Z) }

\newcommand{\Aut}{ {\mathrm{Aut}} }
 
\newcommand{\Hom}{ \mathrm{Hom}}

\newtheorem{theorem}{Theorem}[section]

\newtheorem{proposition}[theorem]{Proposition}
\newtheorem{definition}[theorem]{Definition}
\newtheorem{corollary}[theorem]{Corollary}
\newtheorem{prop-def}[theorem]{Proposition-Definition}
\newtheorem{remark}[theorem]{Remark}

\newtheorem{lemma}[theorem]{Lemme}

\title{Algebraic invariants of orbit configuration spaces in genus zero associated to finite groups}
\author{Mohamad Maassarani}
\date{ }

\begin{document}
\maketitle  

\begin{abstract}
We consider orbit configuration spaces associated to finite groups acting freely by orientation preserving homeomorphisms on the $2$-sphere minus a finite number of points. Such action is equivalent to a homography action of a finite subgroup $G\subset \PGL(\C^2)$ on the complex projective line $\PP^1$ minus a finite set $Z$ stable under $G$. We compute the cohomology ring and the Poincaré series of the orbit configuration space $C_n^G(\PP^1 \setminus  Z)$. This can be seen as a generalization of the work of \cite{Arn} for the classical configuration space $C_n(\C)$ ($(G,Z)=(\{1\},\infty$)). It follows from the work that $C_n^G(\PP^1\setminus Z)$ is formal in the sense of rational homotopy theory. We also prove the existence of an LCS formula relating the Poincaré series of $C_n^G(\PP^1\setminus Z)$ to the ranks of quotients of successive terms of the lower central series of the fundamental group of $C_n^G(\PP^1 \setminus Z)$. The successive quotients correspond to homogenous elements of graded Lie algebras introduced in \cite{MM} and \cite{MMT}. Such formula is also known for classical configuration spaces of $\C$, where fundamental groups are Artin braid groups and the ranks correspond to dimensions of homogenous elements of the Kohno-Drinfeld Lie algebras.
\end{abstract}
\section*{Introduction}
\subsection*{Context and main results}
For $M$ a topological space and $n\geq 1$ an integer, the configuration space of $n$ ordered points of $M$ is the space:
$$ C_n(M)=\{(p_1,\dots,p_n)\in M \vert p_i \neq p_j, \text{for $i\neq j$}\}.$$
Configuration spaces appear naturally in mathematics. For instance, they are encountered in the study braid groups, Knizhnik-Zamolodchikov connections and manifold embeddings. The topology of these spaces and their algebraic invariants are widely studied. \\\\
Let $S$ be an orientable surface. In general, for $S$ with compact boundary (eventually empty), the space $C_n(S)$ is aspheric, except for $S$ homeomorphic to the $2$-sphere $S^2$. The fundamental group $\pi_1C_n(S)$ of $C_n(S)$ is known as Artin pure braid group on $n$ strands for $S=\C$ (\cite{Art}) or generally as the surface pure braid group on $n$ strands (\cite{GuaSPB}) .\\\\
In \cite{Arn}, V. Arnold computed the cohomology ring of $C_n(\C)$ with integer coefficients and its Poincaré series. The cohomology ring of $C_n(S^2)$ was computed in \cite{ziegC} (also considered in \cite{FH}). The spaces $C_n(\C \setminus X)$ where $X$ has cardinal $1$ or $n$ were considered in \cite{Felder} and \cite{var}. In a more general context, Fulton and MacPherson (\cite{FM}) computed a model of $C_n(M)$ (a differential graded algebra whose cohomology is the cohomology of $C_n(M)$) for $M$ a smooth compact complex projective variety. These models were simplified by Kriz (\cite{Kriz}) and then used by Bezrukavnikov (\cite{bezr}) to compute the Malcev Lie algebra $\mathrm{Lie}\pi_1C_n(S)$ of $\pi_1C_n(S)$ for $S$ compact (see also \cite{BE3}, and \cite{Koh} for the case $S=\C$ where the Lie algebra is the Kohno-Drinfeld algebra). \\\\
We will be interested in variants of configurations spaces (introduced in \cite{Xico}, known as orbit configuration spaces:
$$ C_n^H(M)=\{(p_1,\dots,p_n)\in M \vert p_i \neq h\cdot p_j, \text{for $i\neq j$ and $h\in H$}\},$$
where $H$ is a group acting (continuously) on $M$. As for classical configuration spaces the projections $C_n^H(M) \to C_{k}^H(M)$ on the first $k$ coordinates $(k\leq n)$ are locally trivial fibrations under reasonable assumptions (\cite{Xico},\cite{XicoSL}). In \cite{DCoh} and \cite{BE} the authors considered the orbit configuration space $C_n^{\mu_p}(\C^\times)$ where $\mu_p$ is the group of $p$-roots of the unity acting by multiplication on $\C^\times$, and computed using different approaches the Malcev Lie algebra of $\pi_1C_n^{\mu_p}(\C^\times)$. In \cite{XicoCohGauss}, the orbit configuration spaces $C_n^H(\C)$ where $H=\Z+i\Z$ acts additively is studied. In \cite{CKX}, the case of orbit configuration spaces obtained out of a surface subgroup (of genus $g\geq 2$) of $\PSL(\R^2)$ acting freely on the upper half plane $\{z\in \C, \mathrm{Im}(z)>0\}$ is considered. Orbit configuration spaces of spheres with respect to the antipodal action were studied in \cite{XicoAnt1} and \cite{ziegAnt} (see also \cite{XicoAnt2}).\\\\
Here, we consider the orbit configuration spaces $C_n^H(S^2\setminus Y)$ where $Y$ is a finite set and $H$ is a finite group acting freely by orientation preserving homeomorphisms on $S^2\setminus Y$. Under these assumptions, the action of $H$ on $S^2\setminus Y$ is in fact equivalent to the natural action of a finite homography group $G\subset \PGL(\C^2)$ (isomorphic to $H$) on $\PP^1 \setminus Z$, where $\PP^1\simeq S^2$ is the complex projective line and $Z$ is a finite $G$-stable set containing the irregular points of $G$ (points with non-trivial stabilizer). One can give a complete classification of such actions (Cf. Subsection \ref{S1}) and give a biholomorphic equivalence between some orbit configuration spaces associated to isomorphic groups (Cf. Subsection \ref{S21}).\\\\\
In \cite{MM}, we have computed the Malcev Lie algebra of $C_n^G(\PP^1\setminus Z)$ for $Z$ equal to the set of irregular points of $G$ and then extended the work (in \cite{MMT}) to the case of an arbitrary $G$-stable set $Z$ containing the irregular points of $G$. In particular, we recover from \cite{MM} the Lie algebras computed by \cite{DCoh} and \cite{BE} for $(G,Z)=(\{z\mapsto \zeta z\vert \zeta \in \mu_p\}, \{0,\infty\})$ and the Kohno-Drinfeld Lie algebras for $(G,Z)=(\{1\},\infty)$. The work implies that $C_n^G(\PP^1\setminus Z)$ is $1$-formal in the sense of rational homotopy theory. For $(G,Z)\neq (\{1\}, \emptyset)$, the space $C_n^G(\PP^1\setminus Z)$ is biholomorphic to the complement of a hypersurface in $\C^n$ and hence in $\PP^n(\C)$ (Cf. Subsection \ref{S22}) and hence the $1$-formality for $(G,Z)\neq \{1,\emptyset\}$ follows also from \cite{Koh}. \\\\
\textbf{Main results of the paper}\\
Fix $R\subset \C$ a unital ring and set $X_n:=C_n^G(\PP^1\setminus Z)$.\\\\
For $(G,Z)\neq (\{1\},\emptyset)$ we prove (Cf. Subsection \ref{S41} and Section \ref{S5}) that:
\begin{itemize}
\item[1)] The singular cohomology ring $H^*(X_n,R)$ is isomorphic to the $R$-subalgebra $\Omega_D^*(X_n)_R$ of holomorphic (in fact algebraic) closed forms generated by logarithmic $1$-forms $\{\omega_a\}_a$ (having integer periods) corresponding to the irreducible components $\{D_a\}_a$ of $(\PP^1)^n \setminus X_n$. The isomorphism is induced by integration of forms on homology classes.
\item[2)] We find relations between the elements $\{\omega_a\}_a$ of $\Omega_D^*(X_n)_R \simeq H^*(X_n,R)$ wich together with the antisymmetry relations give a presentation of the algebra.
\item[3)] The homology and cohomology groups of $X_n$ with coefficients in $R$ are free $R$-modules of finite type and the Poincaré series $P_{X_n}$ of $X_n$ is given by: $$P_{X_n}(t)=\prod_{k=1}^n(1+\alpha_kt),$$ where $\alpha_k=\vert G \vert (k-1) +\vert Z\vert -1$.
\item[4)] The space $X_n$ is formal in the sense of rational homotopy (Cf. Subsection \ref{S formality}, for the definition).
\item[5)] The space $X_n$ is a $K(\pi,1)$ spaces.
\item[6)] The fundamental group $\pi_1(X_n)$ of $X_n$ is an iterated almost direct product of free groups in the sense of \cite{CohSuc} (Cf. Subsection \ref{S IA} for the definition) and the ranks of the abelian groups $\Gamma_i \pi_1X_n/\Gamma_{i+1} \pi_1 X_n$ corresponding to the lower central series filtration $\{\Gamma_i \pi X_n\}$ of $\pi_1X_n$ can be related to the Poincaré series of $X_n$ by the "LCS formula":
$$P_{X_n}(-t)=\prod_{i\geq 1} (1-t^i)^{\phi_i(\pi_1X_n)},$$
where $\phi_i(\pi_1X_n)$ is the rank $\Gamma_i \pi_1X_n/\Gamma_{i+1} \pi_1 X_n$, and for which we give an explicit formula.
\end{itemize}
For the case $(G,Z)=(\{1\},\emptyset)$, i.e. $X_n=C_n(\PP^1) \simeq C_n(S^2)$, the cohomology ring was computed in \cite{ziegC}. Here we show (Cf. Subsection \ref{S42} and Section \ref{S5}) that:
\begin{itemize}
\item[7)] The space $C_n(\PP^1)$ is formal and construct a subalgebra of closed differential forms isomorphic via integration to $H^*(C_n(\PP^1),R)$ for $\frac{1}{2}\in R$ and $R$ a principal ideal domain.
\item[8)] We have an LCS formula relating the Poincaré series (that factors into a product of linear terms) to the ranks of the abelian groups $\Gamma_i \pi_1X_n/\Gamma_{i+1} \pi_1 X_n$ for which we give an explicit formula.
\end{itemize}
The constants $\phi_i(\pi_1X_n)$ appearing in $(6)$ and the LCS formula mentioned in $(8)$ correspond to the dimension of homogenous elements of graded Lie algebras introduced in \cite{MM} and \cite{MMT} (Cf. Remark \ref{rmk Malcev}).\\\\
Results $(1),(2)$ and $(3)$ are obtained as a generalization of the work of \cite{Arn}. Result $(4)$ follows from $(1)$ using standard facts. Point $(5)$ is obtained by classical means using the homotopy long exact sequence of the fibration associated to the orbit configuration spaces. To prove the existence of the LCS formula we use results from \cite{CohSuc}, \cite{FRLCS}). In fact, using $(5)$ one can also deduce the Poincaré polynomial of $X_n$ from the work of \cite{CohSuc} (Cf. Section \ref{S5}). Results $(7)$ and $(8)$ follow from topological properties known in the literature.\\\\
In the case $G$ cyclic generated by $z\mapsto \zeta z$ where $\zeta$ is a root of the unity and $Z=\{0,\infty\}$ or $G=\{1\}$ and $\vert Z \vert =1$, $C_n^G(\PP^1\setminus Z)$ is the complement in $\C^n$ of a central hyperplane arrangement. The corresponding algebras $\Omega_D^*(X_n)_\Z$  are those of \cite{Bries}, \cite{Arn} and the presentations are equivalent to those given in \cite{OS} (Cf. last paragraph of Subsection \ref{S32}). As seen previously, $\pi_1C_n(\C)$ is the Artin pure braid group on $n$ strands and the LCS formula is known (see for instance \cite{CohSuc}).
\subsection*{Outline of the paper}
Section 1, consists of reminders on formality, De Rham theorem, Iterated almost direct products, their LCS formula and the classification of finite homography groups actions on $\PP^1$. We also state the fact relating orientation preserving finite group actions on a $2$-sphere with finite punctures to the action of finite homography groups actions on $\PP^1$.\\\\
In subsection \ref{S21}, we recall the definition of an orbit configuration space $C_n^H(M)$, the action of $H^n\rtimes \mathfrak{S}_n$ on $C_n^H(M)$ and the fibration theorem. We then consider the orbit configuration spaces $C_n^H(S^2\setminus Y)$ where $Y$ and $H$ are both finite and the action of $H$ is free and preserves the orientation. We deduce from results of section $1$ that $C_n^H(S^2\setminus Y)$ is homeomorphic to $C_n^G(\PP^1\setminus Z)$ for $G$ a finite homorgraphy group (isomorphic to $H$) and $Z$ a $G$-stable finite set containing the irregular points of $G$. We then show that some orbit configuration spaces $C_n^G(\PP^1\setminus Z)$ are biholomorphic. Finally, we examen the fiber $F_n$ of the fibration $C_n^G(\PP^1\setminus Z)\to C_{n-1}^G(\PP^1\setminus Z)$. In subsection \ref{S22}, whe show that $C_n^G(\PP^1\setminus Z)$ is biholomorphic to the complement of a hypersurface in $\C^n$, for $(G,Z)\neq (\{1\},\emptyset)$. In the last subsection (\ref{S23}), we construct generators of $\pi_1C_n^G(\PP^1\setminus Z)$.\\\\
Section \ref{S3} contains three subsection and its results are use in section \ref{S4}. We assume that $(G,Z)\neq (\{1\}, \emptyset)$. In subsection \ref{S31}, we define closed holomorphic $1$-forms $\{\omega_a\}_a$ on $X_n:=C_n^G(\PP^1\setminus Z)$ and consider for $R$ an unital subring of $\C$ the $R$-algebra $\Omega_D^*(X_n)_R$ generated by the forms $\{\omega_a\}_a$, we study the "action" of $G^n\rtimes \mathfrak{S}_n$ (by pullback) on $\Omega_D^*(X_n)_R$ (the algebra is stable under the group). This allows us to deduce (subsection \ref{S32}), from the relations given in \cite{Arn}, relations satisfied in $\Omega_D^2(X_n)_R$ by the $1$-forms $\{\omega_a\}_a$. In section \ref{S32}, we also introduce the algebra $A_n(R)$ quotient of the $R$ exterior algebra $\Lambda (\oplus_a R\cdot \tilde{\omega}_a)$ by analogues of the relations satisfied by $\{\omega_a\}_a$. The relations allow us to define sets of elements spanning both $R$-modules $\Omega_D^*(X_n)_R$ and $A_n(R)$. In the last subsection, we study the periods of the $1$-forms $\{\omega_a\}_a$ ,i.e. values of their integrals on homology classes.\\\\
In section \ref{S4}, we prove the main results $(1),(2),(3),(4)$ and $(7)$. The first four result are proven in subsection \ref{S41} and $(7)$ is proven in \ref{S42}. We use the results on periods of the $1$-forms $\{\omega_a\}_a$ to obtain an integration isomorphism $\phi_{1-}^R: \Omega_{D}^{1-}(X_n)_R \to \Hom_\Z(H_{1-}(X_n,\Z),R)$ ($1-$ is for degree 0 plus degree $1$ parts). We then construct a cohomology extension of the fiber for the fiber bundle $F_n\to X_n \to X_{n-1}$. This is mainly obtained by identifying (isomorphically) a subspace $W_n^* \subset \Omega_{D}^{1-}(X_n)_R$ to $\Hom_\Z(H_{1-}(F_n,\Z),R)$ via integration. Hence, the Leray-Hirsch theorem applies, and we get a splitting $H^*(X_n,R)\simeq H^*(X_{n-1},R)\otimes W_n^*$ and $(3)$ follows. The splitting also allows to give a basis of $H^*(X_n,R)$ close to the generating sets of $\Omega_D(X_n)_R$ and $A_n(R)$. We then deduces using general arguments that integration gives a well-defined $R$-algebra morphism $\Phi_R:\Omega_D^*(X_n)_R \to H^*(X_n,R)$. Results $(1)$ and $(2)$ follow, since the image of the set spanning $\Omega_D^*(X_n)_R$ given in section \ref{S3} is exactly the basis of $H^*(X_n,R)$ obtained from the splittings. Result $(4)$ is obtained from $(1)$ by classical arguments. The remaining case of $(G,Z)=(\{1\}, \infty)$ ($X_n=C_n(\PP^1)$) is studied in subsection \ref{S42}. We use known elements on $C_n(\PP^1)$ and results from subsection \ref{S41} to prove $(7)$ and compute the Poincaré series of $C_n(\PP^1)$. We use a homotopy equivalence between $\PP^1$ and $\C_2(\PP^1)$, a diffeomorphism between $C_n(\PP^1) \simeq C_3(\PP^1) \times C_{n-3}(\PP^1\setminus \{0,1,\infty\})$ (for $n\geq 3$) and diffeomorphisms $C_3(\PP^1)\simeq \PGL(\C^2) \simeq \PP^3(\R)\times \R^3$, where $\PP^3(\R)$ is the real $3$-dimensional projective space. \\\\
Section \ref{S5} contains the proofs of the remaining main results: $(5),(6)$ and $(8)$. They are essentially proved using the homotopy long exact sequence of the fibration $C_{n+1}^G(\PP^1\setminus Z)\to C_n(\PP^1 \setminus Z)$, elements from previous sections and the LCS formula from \cite{CohSuc} for iterated almost direct products of free groups.\\\\
In the appendix, we consider a closed oriented surface $S$. We prove that if $Y\subset S$ is finite (eventually empty) and $H$ is a finite group acting by orientation preserving homeomorphisms on $S\setminus Y$, then the action of $H$ extends to an action on $S$ and there exists a complex structure on $S$ in which $H$ acts holomorphically. We use this result to link finite group actions on $S^2\setminus Y$ to homography actions on $\PP^1$, in section 1.

\tableofcontents
\section{Reminders}
In subsection \ref{S formality}, we recall the definition of formal spaces and some related facts for smooth manifolds. We go rapidly through the construction of the De Rham isomorphism in subsection \ref{S DR}. Subsection \ref{S IA} is devoted to almost product of group and the LCS formula of an iterated almost product of free groups. The last subsection (\ref{S1}) contains reminders on homogarphies of $\PP^1$ including the classification of finite homography groups of $\PGL(\C^2)$. We also relate continuous orientation preserving actions of finite groups on $S^2$ with a finite number of punctures, to actions of finite homography groups on $\PP^1$.
\subsection{Cochain algebras and formality of topological spaces}\label{S formality}
The material in this section and original references can be found in \cite{FelHar} (except for proposition \ref{prop weak gen}). For $R$ a ring an $R$-cochain algebra (dga for short) is a unital associative graded $R$-algebra $A^*=\oplus_{i\geq 0} A^i$, equipped with a differential $d: A^*\to A^*$ mapping $A^i\to A^{i+1}$, such that:
$d(xy)=d(x)y+(-1)^{i}xd(y)$ and $d^2=0$, for $x \in A^i$ and $y \in A^j$. We say that $A^*$ is commutative ($A^*$ is a cdga) if $xy=(-1)^{ij}yx$, for $x$ and $y$ as before. The cohomology of the dga $A^*$ is $H^*(A^*):=\mathrm{Ker}(d)/\mathrm{Im}(d)$ that inherits a structure of graded algebra from $A$. Throughout the rest of the text, a graded map of degree $0$ is simply called graded map or map of graded modules-vector spaces-cdga's.\\\\
We say that two cdga's $A^*$ and $B^*$ are weakly equivalent if there exists a sequence of cdga morphisms (morphisms of algebras respecting the differentials) connecting $A^*$ and $B^*$:
$$A^*\to C(1)^* \leftarrow C(2)^*\to \cdots \to C(n)^*\leftarrow B^*, $$
and inducing isomorphisms in cohomology (quasi-isomorphisms).\\\\
Let $F$ be a field of characteristic zero and $X$ a topological space. One assigns functorially to $X$ the cdga $A_{PL}^*(X)_F$ of polynomial differential forms on $X$ with coefficients in $F$. The algebra $H^*(A_{PL}^*(X)_F)$ is naturally isomorphic to the singular cohomology algebra of $X$ with coefficients in $F$. Moreover, when $X$ is a smooth manifold, $A_{PL}^*(X)_\R$ is weakly equivalent to the cdga $\Omega^*(X)_\R$ of real valued smooth differential forms on $X$ with product the wedge product and differential the exterior differential. Assuming that $X$ have rational homotopy of finite type, we obtain a natural quasi-isomorphism $A_{PL}^*(X)_\Q\otimes F \to A_{PL}^*(X)_F$. It follows that:
\begin{proposition}\label{prop weak gen}
If the rational cohomology of a smooth manifold $X$ is finite dimensional, then $A_{PL}^*(X)_\C$ is weakly equivalent to the cdga $\Omega^*(X)_\C$ of complex valued smooth differential forms on $X$ with product the wedge product and differential the exterior differential.
\end{proposition}
A path connected topological space $X$ is called formal over $F$ (usually over $\Q)$ if $A_{PL}^*(X)_F$ is weakly equivalent to its cohomology (equivalently the singular cohomology of $X$ with coefficients in $F$) with zero differential. It turns out that: if the rational cohomology of $X$ is finite dimensional, then $X$ is formal over $F$ if and only if $X$ is formal over $\Q$.
\begin{comment}
\textbf{Harplin Finite type fomality A equi formal of Atnes K (suillivan). Felix cohomology of finite type APLF= APLQ T F}
\textbf{differential 42, derivation 44, 46 dga= Cohain coalgerba, weak equivalence 115, H(APL) simeq H(sing) algebras (using weak equivalence of cochain algebras) page 126, APL quasi iso to singular as cochain complex (integration) 128-129 et remark page 130 implique la mult en cohomology,   formal 156, weak equilance APL ADR 134} 
\textbf{cohomology of APL and singular cohomology ?}
\end{comment} 
\subsection{De Rham theorem}\label{S DR}
Let $X$ be a smooth manifold and $F$ be the field of real numbers or complex numbers. A smooth singular $k$-simplex on $X$, is a singular $k$-simplex on $X$ admitting a smooth extension to a neighborhood of the standard $k$-simplex $\Delta^k$. One can therefore define the integral of a differential form $\omega$ over a smooth simplex $\sigma^s$ by $\int_{\sigma^s}\omega:=\int_{\Delta_k}(\sigma^{s})^*\omega$, where $\Delta^k$ is endowed with the standard orientation. We denote by $H^*_{DR}(X,F)$ the De Rham cohomology of smooth $F$-valued differential forms on $X$ (corresponding to the cohomology of the cdga $\Omega^*(X)_F$ seen in subsection \ref{S formality}).
One has an integration morphism (\cite{Lee}, \cite{Bred}) of graded $F$-vector spaces:
\begin{align*}
\begin{split}
\int: H_{DR}^*(X,F)&\to \Hom_\Z(H_*(X,\Z),F),\\
\omega &\mapsto ([\sigma] \mapsto \int_{\sigma^s} \omega)
\end{split}
\end{align*}
where $H_*(X,F)$ is the singular homology group of $X$ with coefficients in $F$, $\Hom_\Z$ stands for morphisms of abelian groups, $[\sigma]$ is a homology class and $\sigma^s$ is a smooth representative of $[\sigma]$ (a representative which is a sum of smooth simplices). The natural map $N^F_{X,*}:H^*(X,F)\to \Hom_\Z(H_*(X,\Z),F)$, where $H^*(X,F)$ is the singular cohomology of $X$ with coefficients in $F$, is an isomorphism, since $F$ is a field.
The De Rham theorem states that:
\begin{align*}
\text{$(N^F_{X,*})^{-1}\circ\int: H_{DR}^*(X,F) \to H^*(X,F)$ is an isomorphism of algebras,}
\end{align*}
with $H_{DR}^*(X,F)$ equipped with the wedge product. The theorem is usually stated for $F=\R$. It also holds for $F=\C$, since $H_{DR}^*(X,\C)\simeq H^*_{DR}(X,\R)\otimes \C$. One can find the a proof in \cite{War} (p. 205-207) or \cite{WhGit} (p. 142). It seems that the multiplicative part is not covered in \cite{Lee} and \cite{Bred}.
\subsection{Almost direct products and LCS formula}\label{S IA}
Let $H$ be a group. We denote by the $\{\Gamma_i H\}_{i\geq1}$ the lower central series of $H$: $\Gamma_1 H= H$ and $\Gamma_{i+1} H =(H,\Gamma_i H)$, for $i\geq1$, where $(A,B)$ is the subgroup of $H$ generated by the commutators $(a,b)=aba^{-1}b^{-1}$ for $(a,b)\in A \times B$. Let $H^{ab}$ be the abelianization of $H$ (i.e. $H/\Gamma_2 H$). An $IA$-automorphism of $H$ is an automorphism inducing the identity on $H^{ab}$. A semi-direct product of two groups $H_1\rtimes_\beta H_2$ is called almost direct product if for all $h_2\in H_2$, $\beta(h_2)$ is $IA$.
\begin{proposition}[\cite{FRLCS},\cite{CohSuc}]\label{LCS Formula}
Let $$F:=F(c_n)\rtimes_{\alpha_n} (F(c_{n-1})\rtimes_{\alpha_{n-1}} (F(c_{n-2})\rtimes_{\alpha_{n-1}} (\cdots \rtimes_{\alpha_{3}} (F(c_2) \rtimes_{\alpha_{2}} F(c_1))\cdots ))),$$ be an iterated almost direct product, i.e. the image of $\alpha_i$ consists of IA-automorphisms, for $i\in[1,n-1]$. We have:
$$\prod_{i\geq 1} (1-t^i)^{\phi_i (F)}=\prod_{k=1}^n (1-c_{k}t),$$
where $\phi_i (F)$ is the rank of the abelian group $\Gamma_i F/\Gamma_{i+1}F$ for $i\geq 1$ and $\Gamma_k F$ is the $k$-th term of the lower central series of $F$.
\end{proposition}
The formula follows from the fact that $\Gamma_i F/\Gamma_{i+1}F= \oplus_{l=1}^{m} \Gamma_i F(c_l)/\Gamma_{i+1}F(c_l)$ (\cite{FRLCS}, Theorem 3.1) and the formula (\cite{CombGpTh}, p. 330) known for free groups $(1-ct)= \prod_{i\geq 1} (1-t^i)^{\phi_i (F(c))}$.
An explicit formula for $\phi_i(F(c))$ is known (\cite{CombGpTh}):
$$ \phi_i(F(c))=\frac{1}{i}\sum_{j\vert i} \mu(j)c^{i/j},$$
where $\mu$ is the Möbius function. Therefore, the constants $\phi_i(F)$ in the proposition are given by:
\begin{equation}\label{eq mob}
\phi_i(F)=\sum_{l=1}^n \phi_i(F(c_l))=\frac{1}{i}\sum_{l=1}^n\sum_{j\vert i} \mu(j)c_l^{i/j},
\end{equation}
where $\mu$ is the Möbius function.
\begin{remark}\label{rmk POP}
In \cite{CohSuc}, it is shown that the homology with integer coefficients of $F$ as in the proposition, is free as an abelian group and that the Poincaré series of $F$ is given by $\prod_{k=1}^n (1+c_{k}t)$.
\end{remark}\subsection{Orientation preserving group actions in genus $0$}\label{S1}
Let $a_H: H\times X_H\to X_H$ and $a_{H'}: H'\times X_{H'}\to X_{H'}$ be actions by homeomorphisms of finite groups $H$ and $H'$ on topological spaces $X_H$ and $X_{H'}$. We say that the actions $a_H$ and $a_{H'}$ are equivalent if there exists an isomorphism $f:H\to H'$ and a homeomorphism $g: X_H \to X_{H'}$ such that $ga_{H'}(f\times g)=a_H$. The irregular points of $H$ (with respect to $a_H$) are the elements of $X_H$ with non-trivial stabilizer.\\\\
Recall that the action of $\GL(\C^2)$ on $\C^2$ is compatible to the projection $\C^2\to \PP^1,(z,w)\mapsto [z:w]$ and therefore we have an induced action of $\GL(\C^2)$ on $\PP^1$ given for $A\in \GL(\C^2)$ by:
$$ A\cdot [z:w]=[A_1(z,w): A_2(z,w)] \quad \text{where $A(z,w)=(A_1(z,w) , A_2(z,w))$},$$
The mapping $[z:w]\mapsto A\cdot [z:w]$ is called a homography. It is common to only give the homography in the chart $[z:1]$, i.e. $z\mapsto \frac{A_1(z,1)}{A_2(z,1)}$. The kernel of this action (a) is the Center $C$ of $\GL(\C^2)$ consisting of homotheties and hence the induced action of the projective linear group $\PGL(\C^2)=\GL(\C^2)/C$ on $\PP^1$ is faithful. The group $\PGL(\C^2)$ is also called the homography group of $\PP^1$.\\\\
The following facts are well known. We give proofs or references for completeness.
\begin{proposition}
The group of automorphisms (biholomorphic self maps) of $\PP^1$ is $\PGL(\C^2)$ acting by homorgraphy and every element of $\PGL(\C^2)$ fixes at least a point in $\PP^1$.
\end{proposition}
\begin{proof}
For the first statement one can find a proof in \cite{Lieb} (Cf. p. 83). The second statement follows from the fact that an element of $\GL(\C^2)$ have at least one eigenvector.
\end{proof}
\begin{proposition}\label{prop class}
Let $G$ be a finite subgroup of $\PGL(\C^2)$.
\begin{itemize}
\item[1)] The group $G$ is either cyclic, diherdral or isomorphic to $\mathfrak{A}_4, \mathfrak{S}_4$ or $\mathfrak{A}_5$, and all these groups occur.
\item[2)] Two isomorphic finite subgroups of $\PGL(\C^2)$ are conjugate.
\item[3)] The action of $G$ on $\PP^1$ is equivalent to one of these actions:
\begin{itemize}
\item The action a group $G_R$ generated by a rotation of finite order on the $2$-sphere.
\item The action of $\langle G_R,r \rangle$ on $S^2$, where $G_R$ is as above and $r$ is a rotation of order $2$ with axis orthogonal to the axis of rotation of $G_R$ (different choices of $r$ give equivalent actions).
\item The action of the isometry group of a Platonic solid ($\mathfrak{A}_4, \mathfrak{S}_4$ or $\mathfrak{A}_5$) on the surface of the corresponding solid which is a topological $2$-sphere.
\end{itemize}
\item[4)] Every non-trivial element of $G$ fixes exactly two points.
\item[5)] If $G$ is not trivial, then the number of points with non-trivial stabilizer (irregular points) is $2$ if $G$ is cyclic or $2+\vert G \vert$ otherwise. These points form $2$ orbits if $G$ is cyclic and $3$ orbits otherwise.
\end{itemize}
\end{proposition}
\begin{proof}
The diffeomorphism $S^2\to \PP^1$ obtained from stereographic projections gives an identification between the groups $\mathrm{SO}(\R^3)$ and $\PSU(\C^2)$ (\cite{Beard}, p.63). The facts above hold for $\mathrm{SO}(\R^3)$ (\cite{Mart}, p.184) and one only needs to show that every finite subgroup of $\PGL(\C^2)$ is conjugate to a subgroup of $\PSU(\C^2)$. The inclusion $\PSL(\C^2)\to \PGL(\C^2)$ is an isomorphism and the preimage of a finite subgroup of $\PSL(\C^2)$ in $\SL(\C^2)$ is a finite subgroup. Let $G$ be a finite subgroup of $\SL(\C^2)$ and $f$ be the standard Hermitian form on $\C^2$. The group $G$ is an isometry group for the non-degenerate Hermitian form $f_G:= \sum_{g\in G} g^*f$, where $g^*f$ is the pullback of $f$ by $g$. Both $f$ and $f_G$ are non-degenerate. Hence, we have isomorphism of Hermitian spaces $(\C^2,f) \overset{T}{\to} (\C^2,f_G)$ and the group $TGT^{-1}$ is a subgroup of $\SU(\C^2)$. We have proved the proposition.
\end{proof}
The following proposition is a special case of proposition \ref{prop Fresult} proved in the appendix.
\begin{proposition}\label{prop class ext}
Let $H$ be finite group acting by orientation preserving homeomorphisms on $S^2\setminus Y$, where $Y$ is a finite subset of $S^2$. There exists a finite subgroup $G$ of $\PGL(\C^2)$ such that the action of $G$ by homography on $\PP^1 \setminus Z$, for a given finite $Z \subset \PP^1$ stable under the action of $G$, is equivalent to the action of $H$ on $S^2\setminus Y$.
\end{proposition}
\section{Orbit configuration spaces in genus $0$ associated to finite groups}
In the first subsection, we recall the definition of orbit configuration spaces, a fibration theorem for these spaces, and use the results of subsection \ref{S1}, to show that the orbit configuration space $C_n^H(S^2\setminus Y)$ where $H$ is a finite group acting freely by orientation preserving homeomorphisms on $S^2$ minus a finite set $Y$, is homeomorphic to $C_n^G(\PP^1\setminus Z)$, where $G$ is a homography group and $Z$ is a given finite set. In the second subsection, we prove that the orbit configuration spaces $C_n^G(\PP^1\setminus Z)$ correspond to the complement of a singular hypersurface in $\C^n$ (except when $(G,Z)=(\{1\},\emptyset)$). We then define loops generating the fundamental group of $C_n^G(\PP^1\setminus Z)$ , in subsection \ref{S23}.
\subsection{Orbit configuration spaces of the $2$-sphere}\label{S21}
Let $H$ be a group acting by homeomorphisms on a topological manifold $M$. The orbit configuration space of (ordered) $n$-points of $M$ with respect to $H$ is the topological subspace of $M^n$:
$$ C_n^H(M)= \{(p_1,\dots,p_n)\in M^n\vert p_i\neq H\cdot p_j, \:\:\text{for $1\leq i\neq j \leq n$}\}, $$
for $n\geq 1$, and by convention $C_0^H(M)$ is a point. When $H$ is trivial, the space is the classical configuration space of ordered $n$-points of $M$.\\\\
The space $C_n^H(M)$ is naturally equipped with a topological action of the semidirect product $H^n\rtimes \mathfrak{S}_n$. The action is given by the datum:
$$g_i \cdot (p_1,\dots,p_n)= (p_1,\dots, p_{i-1},g\cdot p_i, p_{i+1},\dots,p_n)$$ $$\text{and} \quad \sigma\cdot (p_1,\dots,p_n)= (p_{\sigma^{-1}(1)},\dots,p_{\sigma^{-1}(n)}), $$
for $i\in[1,n], g\in G,\sigma \in \mathfrak{S}_n$ and where $g_i=(1,\dots,1,g,1,\dots,1)$ with $g$ at the $i$-th position.
\begin{theorem}[\cite{Xico}]\label{prop fib}
If $H$ is a finite group acting freely by homeomorphism on a boundaryless manifold $M$. For $n\geq k \geq 0$, the projection $C_n^H(M)\to C_k^H(M)$ on the first $k$ coordinates is a locally trivial fibration.
\end{theorem}
Here we study orbit configuration spaces associated to a finite group $H$ acting freely by orientation preserving homeomorphisms on the $2$-sphere $S^2\simeq \PP^1$ minus a finite number of points:
$$C_n^H(S^2\setminus Y)=\{(p_1,\dots,p_n)\in (S^2\setminus Y)^n\vert p_i\neq H\cdot p_j, \:\:\text{for $1\leq i\neq j \leq n$}\}, $$
where $Y\subset S^2$ is finite. It follows from proposition \ref{prop class} and proposition \ref{prop class ext}, that:
\begin{proposition}\label{prop bir}
The space $C_n^H(S^2\setminus Y)$ for $n,H$ and $Y$ as in the previous paragraph, is homeomorphic to $C_n^G(\PP^1\setminus Z)$, where $G\simeq H$ is a finite subgroup of $\PGL(\C^2)$ acting naturally (by homographies) on $\PP^1$ and $Z\subset \PP^1$ is finite, stable under the action of $G$, and contains the irregular points of $G$. Moreover, $Z\neq \emptyset$ if $G\neq \{1\}$.
\end{proposition}
In the sequel $G$ and $Z$ are as in the above proposition.
\begin{remark}\label{rmk infty}
Let $G$ be a subgroup of $\PGL(\C^2)$ and $h$ a homorgraphy. The map $h^{\times n}$ induces a biholomorphism $C_n^G(\PP^1\setminus Z) \to C_n^{hGh^{-1}}(\PP^1 \setminus h(Z))$. In particular, for $(G,Z)\neq (\{1\}, \emptyset)$ we can assume up to applying a biholomorphism that $\infty \in Z$ (as an irregular point if $G\neq \{1\}$). Moreover, two configuration spaces of the form $C_n^G(\PP^1\setminus Z)$ associated to two isomorphic subgroups of $\PGL(\C^2)$ can be biholomorphic depending on the choice of the stable sets, since two isomorphic subgroups of $\PGL(\C^2)$ are conjugate.
\end{remark}
If $G=\{1\}$, the spaces $C_n^G(\PP^1\setminus Z)$ correspond to classical configuration space : the configuration space of the sphere if $Z$ is empty, the configuration space of the plane if $\vert Z \vert =1$, the configuration space of the plane minus $1$ point if $\vert Z \vert=1$ etc..\\\\
The space $\PP^1\setminus Z$ and the action of $G$ satisfy the conditions of proposition \ref{prop fib}. Hence, for $n\geq 1$, we have a fiber bundle (the base is paracompact):
$$ F_n \to C_n^G(\PP^1\setminus Z) \overset{\pi_n}{\to} C_{n-1}^G(\PP^1\setminus Z)$$
where $\pi_n$ is the projection on the first $n-1$ coordinates and the fiber $F_n$ is a $2$-sphere minus $\vert Z \vert + (n-1) \vert G \vert$ points. The orbit configuration space $C_n^G(\PP^1\setminus Z)$ is path connected.
\subsection{Orbit configuration spaces of the $2$-sphere as complements of hypersurfaces}\label{S22}
Take $G$ and $Z$ as in the previous subsection. The space $C_n^G(\PP^1\setminus Z)$ is the complement in $(\PP^1)^n$ of the hypersurfaces:
$$D_{ij}(g)= \{(p_1,\dots,p_n) \in (\PP^1)^n \vert p_i=g\cdot p_j\} \quad D_{kk}(p)=\{(p_1,\dots,p_n) \in (\PP^1)^n \vert p_k=p\}$$
for $i,j,k\in [1,n],g\in G$ and $p\in Z$ with $i<j$ (to avoid repetitions). Putting coordinates $([z_1:w_1],\dots, [z_n:w_n])$ on $(\PP^1)^n$, setting $g([z:w])=[g_1(z,w):g_2(z,w)]$ for $g\in G$, and $p=[p_z:p_w]$ for $p\in Z$, we get that $D_{ij}(g)$ and $D_{ii}(p)$ are respectively the zero loci of the polynomials :
$$w_ig_1(z_j,w_j) -z_i g_2(z_j,w_j) \quad \text{and}\quad w_ip_z -z_ip_w.$$
Up to a biholomorphism, we can assume that (Cf. Remark \ref{rmk infty}): $\infty \in Z$ and that the other elements of $Z$ lie in $\C$. Under these assumptions $C_n^G(\PP^1\setminus Z)$ is the complement in $\C^n$ of the hypersurface of equation:
$$ f(z_1,\dots,z_n)= \underset{1\leq k \leq n, q\in Z\setminus \{\infty\}}{\prod} P_{ii}^q(z_1,\dots,z_{n}) \underset{1\leq i<j \leq n, g\in G}{\prod} P_{ij}^\alpha(z_1,\dots,z_{n})=0,$$
where $ P_{ii}^q=z_i-q \quad \text{and} \quad P_{ij}^g=g_1(z_j,1)-z_ig_2(z_j,1)$.\\\\
If $G$ is the cyclic of order $m\geq 1$. We can assume, up to conjugacy-biholomorphism (Cf. Remark \ref{rmk infty}), that $G$ is generated by $z\mapsto \zeta z $ where $\zeta$ is a $m$-th primitive root of the unity. In that case, the hypersurface is given by the equation:
\begin{align*}
f(z_1,\dots,z_{n+1})&= \underset{\underset{ q \in (Z\setminus \{\infty\})}{1\leq i \leq n}}{\prod} (z_i-q) \underset{\underset{0\leq k <m}{1\leq i<j \leq n}}{\prod} (z_i-\zeta^k z_j)=0.
\end{align*}
This arrangement is central (the intersection of all the hyperplanes is not empty) if and only if $Z=\{0,\infty\}$. In the case $G$ not cyclic, the "collection" of the irreducible components of the hypersurface $f=0$ can not be mapped homeomorphically to an arrangement of hyperplanes. Indeed, if $G$ is not cyclic, there exists a $g\in G$ not stabilizing the infinity. The intersection $\{P_{ij}^1=0\}\cap \{P_{ij}^g=0\}$ for such a $g$ contains exactly two points corresponding to the points fixed by $g$. Hence, $\{P_{ij}^1=0\}\cap \{P_{ij}^g=0\}$ is not homeomorphic to the intersection of $2$ hyperplanes, since the latter is always connected.

\subsection{Generators of the fundamental group}\label{S23}
We will construct loops of $\Cng$, for $n\geq 1$. We fix a base point $P_n=(p_1,\dots,p_n)\in \Cng$ and we set for $i\in[1,n]$:
$$Y_i= \Pec \cup Y_{i,G}, \text{ where } Y_{i,G}=\cup_{j\neq i\in [1,n], } G\cdot p_j.$$
The space $\PP^1\setminus Y_i$ is naturally homeomorphic to the fiber of $\pi_n:C_{n+1}^G(\PP^1\setminus Z) \to \Cng$ over $P_n$. We endow $\PP^1$ with its natural orientation. For $q \in Y_i$, let $\gamma_i(q)$ bee a smooth simple anticlockwise oriented loop of $\PP^1$ based at $p_i$, avoiding $q$ and bounding a closed disc $D(q)$ such that $D(q)\cap Y_i=q$.\\\\
We choose the loops $\gamma_i(q)$ so that they generate the fundamental group of $\PP^1\setminus Y_i$ based at $p_i$.
\begin{definition}
For $i,j\in[1,n], g\in G$ and $p \in \Pec$ we define the smooth loops $x_{ij}^g$ (for $i\neq j$) and $x_{ii}^p$ of $C_n^G(\Pe)$ based at $P_n$ by the following:
$$ x_{ij}^g(t)=(p_1,\dots,p_{i-1}, \gamma_i(g\cdot p_j)(t), p_{i+1},\dots,p_n),$$
$$x_{ii}^p(t)=(p_1,\dots,p_{i-1}, \gamma_i(p)(t), p_{i+1},\dots,p_n), $$
for $t\in[0,1]$.\\
\end{definition}
\begin{proposition}\label{gen loop}
For any $p_\infty \in Z$, the loops $x_{ij}^g$ and $x_{kk}^p$ for $g\in G, p\in Z$ and $ i, j, k \in [1,n]$ such that $i< j$ and $p\neq p_\infty$, generate $\pi_1(\Cng,P_n)$.
\end{proposition}
\begin{proof}
The result can be obtained by induction using the long exact sequences of the fibrations $C_k^G(\PP^1\setminus Z) \to C_{k-1}^G(\PP^1\setminus Z)$ (see for instance proposition 3.3 of \cite{MM}).
\end{proof}
We have seen that $\Cng$ is biholomorphic to $\C^n \setminus \{f=0\}$, where $\{f=0\}$ is a hypersurface. The loops in the proposition are meridian loops with respect to irreducible components of the hypersurface $\{f=0\}$ and hence generate $H^1(\Cng, \Z)$ (\cite{HG}, p. 455). Moreover, they generate $H^1(\Cng,\Z)$ freely (\cite{DimB}, p. 102). This will also follow from other elements in the next sections.

\section{Differential forms on the orbit configuration spaces}\label{S3}
In subsection \ref{S31}, we introduce, for $(G,Z)\neq (\{1\},\emptyset)$, closed holomorphic $1$-forms on $X_n:=C_n^G(\PP^1\setminus Z)$ and consider, for $R\subset \C$ an unital ring, the $R$-algebra $\Omega_D^*(X_n)_R$ of differential forms on $X_n$ generated by these forms. We also study the action of $G^n\rtimes \mathfrak{S}_n$ on $\Omega_D^*(X_n)_R$. The action of $G^n\rtimes \mathfrak{S}_n$ is used in subsection \ref{S32} to deduce relations (defined over $\Z$) in $\Omega_D^2(X_n)_R$. The relations allow us to give a family of forms spanning $\Omega_{D}^*(X_n)_R$ as an $R$-module and to define, by generators and relations, an $R$-algebra $A_n(R)$ equipped with a natural surjective map to $\Omega_D^*(X_n)_R$. In the last subsection, we study the periods of $\Omega_D^1(X_n)_R$ (values of integrals of elements of $\Omega_D^1(X_n)_R$ on $1$-homology classes of $X_n$). The results on periods will be used in the next section.\\\\
Till the end of this section, we assume that $(G,Z)$ is different from $(\{1\},\emptyset)$.
\subsection{Differential forms and the algebra $\Omega_D^*(C_n^G(\PP^1\setminus Z))_R$}\label{S31}
For $A \in \PGL(\C^2)$ with matrix $\begin{pmatrix}a& b\\c & d\end{pmatrix}$ in the canonical basis of $\C^2$, set:
$$P_A(x,y)=(cy+d)(x-A\cdot y)=x(cy+d)-(ay+b),$$ where $A\cdot y=\frac{ay+b}{cy+d}$. We assign to $h\in \PGL_2(\C)$ a 1-form
\begin{align}\label{omegah}
\omega^{h}(x,y):=d\log(P_{A}(x,y))=\frac{dP_A(x,y)}{P_A(x,y)}, \end{align}
where $A$ is a lift of $h$ to $ \PGL(\C^2)$. The definition of $\omega^{h}(x,y)$ do not depend on the choice of the lift, since $P_{\lambda A}=\lambda P_{A}$ for $\lambda$ a scalar. The form $\omega^h(x,y)$ is holomorphic over $\C^2\setminus \{(x,y)\in\C^2 \vert x=h(y)\}$. Indeed, the set of zeros of the polynomial $P_A$ is exactly the set $ \{(x,y)\in\C^2 \vert x=h(y)\}$. We also define for $p \in \C$ the holomorphic $1$-form on $\C \setminus \{p\}$:
\begin{align}\label{omegap}
\omega^p(x):=d\log(x-p).
\end{align}
Both $\omega^p(x)$ and $\omega^{h}(x,y)$ are closed forms.
\\\\
We pick a $p_\infty \in Z$ and chose a homography $h_Z \in \PGL(\C^2)$ mapping $p_\infty$ to $\infty$. If $Z$ contains $\infty$ we take $p_\infty=\infty$ and $h_Z=\mathrm{id}$. The map $h_Z^{\times n}$ induces a biholomorphism $C_n^G(\PP^1\setminus Z) \to C_n^{h_ZGh_Z^{-1}}(\PP^1\setminus h_Z(Z))$ and $\PP^1 \setminus h_{Z}(Z)$ is $\C\setminus (Z\setminus \infty)$. We will introduce differential forms on $C_n^G(\PP^1\setminus Z)$, using the coordinates $([z_1: 1],\dots, [z_n:1])$. The forms will depend on the choice of $p_\infty$ and $h_Z$, but as we will see later on the algebra of complex valued forms generated by these forms is independent of these choices.
\begin{definition}\label{def forms}
\begin{itemize}
\item[1)] For $g\in G, p \in Z\setminus \{p_\infty\}$ and $i,j,k\in [1,n]$, with $i<j$, we define the holomorphic closed forms $\omega_{ij}^g$ and $\omega_{ii}^p$ of $C_n^G(\PP^1\setminus Z)$, by:
$$\omega_{kk}^p=\frac{1}{2i\pi} (h_Z^{\times n})^* \omega^{h_Z(p)}(z_k) \quad \text{and} \quad \omega_{ij}^g=\frac{1}{2i\pi} (h_Z^{\times n})^*\omega^{h_Zgh_Z^{-1}}(z_i,z_j) ,$$
where the star $*$ is for pullback, $\omega^{h_Z(p)}$ and $ \omega^{h_Zgh_Z^{-1}}$ are as in (\ref{omegap}) and (\ref{omegah}), and $\omega_i^{\infty}=0$ (and hence $\omega_i^{p_\infty}=0$) by convention.
\item[2)] For $R$ an unital subring of $\C$, we define $\Omega_D^*(C_n^G(\PP^1\setminus Z))_R$ to be the $R$-subalgebra of complex valued forms on $C_n^G(\PP^1\setminus Z)$, generated by the $1$-forms defined in $(1)$ with grading induced by the degree of forms.
\end{itemize}
\end{definition}
For instance $\omega_{ij}^1=\frac{1}{2i\pi}d\log(z_i-z_j)$ if $h_Z=\mathrm{id}$.
\begin{lemma}\label{PAB}
For $A$ and $B\in \GL(\C^2)$. We have:
$$ P_A(x,B\cdot y)= \frac{P_{AB}(x,y)}{D_B(y)} \quad \text{and} \quad P_A(B\cdot x, y)=\mathrm{det}(B)\frac{P_{B^{-1}A}(x,y)}{D_B(x)},$$
for $B\cdot z=\frac{b_{1,1}z+b_{1,2}}{b_{2,1}z+b_{2,2}}$ and $D_B(z)=b_{2,1}z+b_{2,2}$, where $b_{i,j}$ $(i,j\in[1,2])$ is the $i,j$ entry of the matrix of $B$ in the canonical basis of $\C^2$ are nonzero complex numbers.
\end{lemma}
\begin{proof}
To prove the lemma one can assume that the matrices of $A$ and $B$ in the canonical basis are $\begin{pmatrix}a& b\\c & d \end{pmatrix}$ and $\begin{pmatrix}a'& b'\\c' & d' \end{pmatrix}$, then compute both sides of each equation.
\end{proof}
Recall that the group $G^n\rtimes \mathfrak{S}_n$ acts on $C_n^G(\PP^1\setminus Z)$. The action is holomorphic and hence the group acts on the right (by pullback) on complex valued differential forms on $C_n^G(\PP^1\setminus Z)$.
For $h$ a homography of $\PP^1$ and $i \in [1,n]$, we denote by $h_i$ the bijection $(\PP^1)^n\to (\PP^1)^n$ acting by $h$ on the $i$-th coordinate and acting trivially on the other coordinates.
\begin{proposition}\label{lem act}
For $g,h,f,\in G, p\in (Z\setminus \{p_\infty\}), \sigma \in \mathfrak{S}_n$ and $i,j,k,l\in[1,n]$ with $i\neq j$, we have:
$$ (h_if_j)^*\omega_{ij}^{g}=\omega_{ij}^{h^{-1}gf}-\omega_{ii}^{h^{-1}(p_\infty)}-\omega_{jj}^{f^{-1}( p_\infty)}, \quad h_k^*\omega_{kk}^p=\omega_{kk}^{h^{-1}( p)}-\omega_ {kk}^{h^{-1}(p_\infty)} ,$$
$$h_k^*\omega_{ij}^{g}=\omega_{ij}^{g}\quad \text{if $k\neq i,j$},\quad h_l^*\omega_{kk}^p=\omega_{kk}^p,$$
if $l\neq k $, and
$$ \sigma^*\omega_{ij}^g=\omega_{\sigma^{-1}(i) \sigma^{-1}(j)}^g, \quad \sigma^*\omega_{kk}^p=\omega_{\sigma^{-1}(k) \sigma^{-1}(k)}^p.$$
\end{proposition}
\begin{proof}
One deduces the proposition from the case $\infty \in Z$ (i.e. $p_\infty=\infty$ and $h_Z=\mathrm{id}$), by pulling back the equations. We hence only prove the proposition for $\infty \in Z$. For $v\in \{f,g,h\}$, we chose a lift $\tilde{v}$ of $v$ to $\GL(\C^2)$ . Using the definition of $\omega_{ij}^g$, the definition of a pullback, then by applying lemma \ref{PAB} twice and identifying the terms we get:
\begin{align*}(2i\pi)(h_if_j)^*\omega_{ij}^g&=d\log(P_{\tilde{g}}(\tilde{h}\cdot z_i,\tilde{f}\cdot z_j))\\&=d\log(P_{\tilde{h}^{-1}\tilde{g}\tilde{f}}(z_i,z_j))-d\log(D_{\tilde{f}}(z_j))-d\log(D_{\tilde{h}}(z_i))\\
&=(2i\pi)(\omega_{ij}^{h^{-1}gf}-\omega_ {ii}^{h^{-1}\cdot \infty}-\omega_{jj}^{f^{-1} \cdot \infty}).
\end{align*}
This proves the first equation of the proposition. We now prove the second equation of the proposition. The equation is true for $p=\infty$. Assume $p\neq \infty$. Setting $f=g=1$ in the first equation of the proposition, we get $h_i^*\omega_{ij}^1=\omega_{ij}^{h^{-1}}-\omega_{ii}^{h^{-1}\cdot \infty}$ and then by setting $z_j=p$, we find $h_i^*\omega_{ii}^p=\frac{1}{2i\pi}d\log(P_{\tilde{h}^{-1}}(z_i,p))-\omega_{ii}^{h^{-1}\cdot\infty}$.
The second equation of the proposition (for $p\neq \infty$) follows, since $$P_{\tilde{h}^{-1}}(z_i,p)=\begin{cases}(cp+d)(z_i-h^{-1}\cdot p)& \text{if $h^{-1}\cdot p \neq \infty $}\\ ap+b& \text{otherwise}\end{cases}=(2i\pi)\omega_{ii}^{h^{-1} \cdot p},$$
if $\tilde{h}^{-1}=\begin{pmatrix}a& b\\c & d\end{pmatrix}$. We have proved the second equation of the proposition. The remaining equations are straightforward. We have proved the proposition.
\end{proof}
\begin{corollary}
The algebra $\Omega_D^*(C_n^G(\PP^1\setminus Z))_R$ is stable under the action of $G^n\rtimes \mathfrak{S}_n$.
\end{corollary}

\begin{proposition}
Algebras $\Omega_D^*(C_n^G(\PP^1\setminus Z))_R$ obtained for different choices of $p_\infty$ and $h_Z$ are equal.
\end{proposition}
\begin{proof}
Recall that if $\infty \in Z$, then we take $h_Z=\mathrm{id}$. Assume $\infty \notin Z$ and fix a choice of $p_\infty$ and $h_Z$. Adapting the computations used to prove the first two equations of the previous proposition, by taking $h=f=h_Z$ and $p=h_Z(q)$, we get:
$$ \omega_{ij}^g=\omega^g(z_i,z_j)-\omega^{p_\infty}(z_i)-\omega^{p_\infty}(z_j) \quad \text{and} \quad \omega_{kk}^q=\omega^q(z_k)-\omega^{p_\infty}(z_k),$$
where $\omega^g, \omega^{p_\infty}$ and $\omega^q$ are as in $(\ref{omegah})$ and $(\ref{omegap})$. Given a $q_0 \in Z$, we can add to the generators above $-\omega_{ll}^{q_0}$ for an (or two) appropriate $l$('s) (we add nothing if $q_0=p_\infty$), to deduce that the algebra $\Omega_D^*(C_n^G(\PP^1\setminus Z))_R$ is generated by the forms:
$$\omega^g(z_i,z_j)-\omega^{q_0}(z_i)-\omega^{q_0}(z_j) \quad \text{and} \quad \omega^q(z_k)-\omega^{q_0}(z_k),$$
for $q\in Z,g\in G$ and $i,j,k\in[1,n]$, with $i\neq j$. This proves the proposition.
\end{proof}
We see (from the proof above) that we can give a definition of forms generating $\Omega_D^*(C_n^G(\PP^1\setminus Z))_R$ without using $h_Z$. Anyway, the pullback formula by $h_Z$ seems to be easier to manipulate.
\subsection{Relations in $\Omega_D^*(C_n^G(\PP^1\setminus Z))_R$ and the algebra $A_n(R)$}\label{S32}
Recall that for $i,j,k $ three distinct integers in $[1,n]$, we have the following relation (\cite{Arn}):
$$\omega_{ij} \wedge \omega_{jk} +\omega_{jk} \wedge \omega_{ik} +\omega_{ik} \wedge \omega_{ij} =0,$$
where $\omega_{st}=d \log(z_s-z_t)$.
%$$d\log(z_i-z_j)\wedge d\log(z_j-z_k)+ d\log(z_j-z_k)\wedge d\log(z_i-z_k)+d\log(z_i-z_k)\wedge d\log(z_i-z_j)=0.$$
\begin{proposition}\label{prop rel}
For $1\leq i \neq j \leq n ,h, g\in G$ and $p,q\in Z\setminus \{\infty\}$, we have:
\begin{equation}\label{rel1}
\omega_{ii}^p \wedge \omega_{ii}^q=0\quad ,\quad \omega_{ij}^g=\omega_{ji}^{g^{-1}},
\end{equation}
\begin{equation}\label{rel2}
\omega_{ij}^h\wedge \omega_{jj}^p=\omega_{ij}^h\wedge\omega_{ii}^{h\cdot p } + \omega_{ii}^{h\cdot p} \wedge \omega_{jj}^p+\omega_{ii}^{h\cdot p_\infty} \wedge \omega_{ij}^h,\end{equation}
\begin{equation}\label{rel3}
\omega_{ik}^h\wedge \omega_{jk}^g=\omega_{ij}^{hg^{-1}}\wedge ( \omega_{jk}^g-\omega_{ik}^h)+ \omega_{jj}^{g\cdot p_\infty} \wedge(\omega_{ij}^{h^{-1}g} -\omega_{jk}^g)+ \omega_{ii}^{h\cdot p_\infty}\wedge (\omega_{ik}^h -\omega_{ij}^{hg^{-1}})+\omega_{ii}^{h\cdot p_\infty}\wedge\omega_{jj}^{g\cdot p_\infty},\end{equation}
for $k\in[1,n]\setminus \{i,j\}$, and
\begin{equation}\label{rel4}
\omega_{ij}^h\wedge \omega_{ij}^g= (\omega_{ii}^{p_1}+\omega_{ii}^{p_2})\wedge (\omega_{ij}^g-\omega_{ij}^h)- \omega_{ii}^{g\cdot p_\infty} \wedge \omega_{ij}^g + \omega_{ii}^{h\cdot p_\infty}\wedge \omega_{ij}^h,
\end{equation}
for $h\neq g$, with $p_1$ and $p_2$ the two points fixed by $hg^{-1}$.
\end{proposition}
\begin{proof}
The first equation in (\ref{rel1}) is straightforward. We prove the second one. Using the definition of the polynomial $P_A$, we get:
$$P_A(y,x)=y(cx+d)-(ax+b)=-(x(-cy+a)-(dy-b))=-\mathrm{det}(A) P_{A^{-1}}(x,y).$$
This proves that $\omega_{ij}^g=\omega_{ji}^{g^{-1}}$. We now prove (\ref{rel2}). Pulling back the relation given before the proposition by $h_Z^{\times n}$, we obtain: $$\omega_{ij}^1\wedge \omega_{jk}^1+\omega_{jk}^1 \wedge \omega_{ik}^1+\omega_{ik}^1\wedge \omega_{ij}^1=0.$$
Setting $z_k=p$, we get: $$\omega_{ij}^1\wedge \omega_{jj}^p+\omega_{jj}^p \wedge \omega_{ii}^p+\omega_{ii}^p\wedge \omega_{ij}^1=0,$$
which is true for $p=p_\infty$. Pulling back the last equation by $h_i^{-1}$, we find using proposition \ref{lem act}:
$$(\omega_{ij}^h-\omega_{ii}^{h\cdot p_\infty})\wedge \omega_{jj}^p+\omega_{jj}^p \wedge (\omega_{ii}^{h\cdot p}-\omega_{ii}^{h\cdot p_\infty})+(\omega_{ii}^{h\cdot p}-\omega_{ii}^{h\cdot p_\infty})\wedge (\omega_{ij}^{h}-\omega_{ii}^{h\cdot p_\infty})=0.$$
This proves (\ref{rel2}). Applying proposition \ref{lem act} to the pullback by $h_i^{-1}$ of the equation $\omega_{ij}^1\wedge \omega_{jk}^1+\omega_{jk}^1 \wedge \omega_{ik}^1+\omega_{ik}^1\wedge \omega_{ij}^1=0$, we find:
$$ \omega_{ik}^h\wedge \omega_{jk}^1=\omega_{ij}^h\wedge \omega_{jk}^{1}+\omega_{ik}^h\wedge \omega_{ij}^{h}+ \omega_{ii}^{h\cdot p_\infty}\wedge (\omega_{ik}^h-\omega_{ij}^h).$$
Pulling this equation by $g_j^{-1}$, we get:
\begin{align*} \omega_{ik}^h\wedge (\omega_{jk}^g -\omega_{jj}^{g\cdot \infty})=\begin{split} &(\omega_{ij}^{hg^{-1}}-\omega_{jj}^{g\cdot p_\infty})\wedge (\omega_{jk}^g -\omega_{jj}^{g\cdot p_\infty})+\omega_{ik}^h\wedge (\omega_{ij}^{hg^{-1}}-\omega_{jj}^{g\cdot p_\infty})\\
&+ \omega_{ii}^{h\cdot p_\infty}\wedge (\omega_{ik}^h -\omega_{ij}^{hg^{-1}}+\omega_{jj}^{g\cdot p_\infty}).\end{split}
\end{align*}
Simplifying this equation gives (\ref{rel3}).
Replacing $z_j$ with $z_i$ in (\ref{rel3}), for $h\neq g$, we get:
\begin{align*} \omega_{ik}^h\wedge \omega_{ik}^g= \omega_{ii}^{hg^{-1}}\wedge (\omega_{ik}^g-\omega_{ik}^h)- \omega_{ii}^{g\cdot p_\infty} \wedge \omega_{ik}^g + \omega_{ii}^{h\cdot p_\infty}\wedge \omega_{ik}^h,\end{align*}
where $\omega_{ii}^{hg^{-1}}$ is the form obtained by replacing $z_j$ with $z_i$ in $\omega_{ij}^{hg^{-1}}$.
For $a\in \PGL(\C^2)$ of finite order and $A$ a lift of $a$ to $\GL(\C^2)$ the polynomial $P_A(x,x)$ admits simple roots corresponding to the elements $x\in \C$ fixed by $a$. Hence, using the convention $\omega^{\infty}(x)=0$, we have:
$$d \log (P_A(x,x))=\omega^{x_1}(x)+\omega^{x_2}(x),$$
where $x_1,x_2\in \PP^1$ are the two fixed points of $a \in \PGL(\C^2)$ and $\omega^{x_i}(x)=d\log(x-x_i)$ as in equation \ref{omegap}. From this, and the definition of $\omega_{ii}^p$ (for $p\in Z$) we deduce that:
$$ \omega_{ii}^{hg^{-1}}=\omega_{ii}^{p_1}+\omega_{ii}^{p_2},$$
where $p_1$ and $p_2$ are the fixed points of $hg^{-1}$ (and $\omega_{ii}^{p_\infty}=0$ by convention). Replacing $\omega_{ii}^{hg^{-1}}$ by $\omega_{ii}^{p_1}+\omega_{ii}^{p_2}$ in the equation obtained previously, we get:
\begin{align*} \omega_{ik}^h\wedge \omega_{ik}^g= (\omega_{ii}^{p_1}+\omega_{ii}^{p_2})\wedge (\omega_{ik}^g-\omega_{ik}^h)- \omega_{ii}^{g\cdot p_\infty} \wedge \omega_{ik}^g + \omega_{ii}^{h\cdot p_\infty}\wedge \omega_{ik}^h.\end{align*}
We obtain $(\ref{rel4})$ by replacing $k$ with $j$ in the last equation. We have proved the proposition.
\end{proof}
\begin{corollary}
Let $\omega^{\alpha}_{ij}$ and $ \omega^{\beta}_{kl}$ be forms as in definition \ref{def forms}. The product $\omega^{\alpha}_{ij}\wedge \omega^{\beta}_{kl}$, is equal to the sum of products $\varepsilon \omega_{rs}^{\alpha'}\wedge \omega_{tu}^{\beta'}$, where $\varepsilon \in \{1,-1\}$, $\omega_{rs}^{\alpha'}$ and $\omega_{tu}^{\beta'}$ are as in definition \ref{def forms}, $r\leq s, t\leq u$ and $ s<u\leq \mathrm{max}(i,j,k,l)$.
\end{corollary}
\begin{corollary}\label{cor base}
The algebra $\Omega_D^*(C_n^G(\PP^1\setminus Z))_R$ is generated as an $R$-module by the forms:
$$\omega_{i_1j_1}^{\beta_1}\wedge \cdots \wedge \omega_{i_kj_k}^{\beta_k}, \quad 0\leq k \leq n , \quad 1\leq i_k\leq j_k \leq n , \quad j_1<j_2<\cdots <j_k $$
with $\beta_k\in G$ if $i_k\neq j_k$ and $\beta_k\in Z\setminus \{p_\infty\}$ if $i_k=j_k$ (for $k=0$ the product is equal to $1$).
\end{corollary}
\begin{definition}
Define $A_n(R)$ as the quotient of the $R$-exterior algebra generated by the elements $\tilde{\omega}_{ij}^g$ and $\tilde{\omega}_{kk}^p$ for $i,j,k\in[1,n]$, with $i\neq j$, $g\in G$ and $p\in Z\setminus\{p_\infty\}$, by the ideal corresponding to relations analogue to those of proposition \ref{prop rel}.
\end{definition}

\begin{proposition}\label{prop Psi}
\begin{itemize}
\item[1)] We have a surjective morphism of graded $R$-algebra $\Psi_R: A_n(R) \to \Omega_D^*(X_n)_R$ given by $\tilde{\omega}_{ij}^g \mapsto \omega_{ij}^g$ and $\tilde{\omega}_{kk}^p\mapsto \omega_{kk}^p$, for $i,j,k\in[1,n], g\in G$ and $p\in Z\setminus\{p_\infty\}$.
\item[2)] The analogue of corollary \ref{cor base} holds for $A_n(R)$.
\end{itemize}
\end{proposition}
For $(G,Z)=(\{1\}, \{\infty \})$, or $(G,Z)=(\langle \zeta z \rangle, \{0,\infty\})$ the space $X_n$ is the complement in $\C^n$ of a central hyperplane arrangement (see subsection \ref{S22}) and its cohomology ring with integral coefficients is isomorphic to $\Omega_D^*(X_n)_\Z$ (\cite{Arn}, \cite{Bries} lemma 5) and a presentation (definition by generators and relations) is known (\cite{Arn}, \cite{OS}). The algebra $A_n(\Z)$ is the algebra $A_n$ of \cite{Arn}. As we will see later, for all $X_n$, $A_n(\Z)$ is isomorphic to $\Omega_D^*(X_n)_\Z$. It follows easly that for $(G,Z)$ as in the beginning of the paragraph, the relations defining $A_n(\Z)$ are alternatives (or equal) to those in \cite{OS} for the corresponding central hyperplane arrangement.
\subsection{Periods of $\Omega_D^1(C_n^G(\PP^1\setminus Z))_R$ and $H^1(C_n^G(\PP^1\setminus Z),R)$}\label{S33}
We recall that for $P_n\in C_n^G(\PP^1\setminus Z)$, we have defined loops $x_{ij}^\alpha$ based at $P_n$, for $i\leq j \in[1,n]$ with $\alpha \in G$ if $i\neq j$ and $\alpha\in Z$ if $i=j$.
\begin{proposition}\label{prop int}
For $\alpha,\beta \in (G\cup Z \setminus \{p_\infty\})$ and $i\leq j,k\leq l \in [1,n]$, such that $x_{ij}^\alpha$ and $\omega_{kl}^\beta$ are well-defined, we have:
$$ \int_{x_{ij}^\alpha} \omega_{kl}^\beta= \delta_{ik} \delta_{jl} \delta_{\alpha \beta},$$
where $\delta_{ab}=1$ if $a=b$ and $\delta_{a,b}=0$ otherwise.
\end{proposition}
\begin{proof}
The formula follows from the case $h_Z=\mathrm{id}$ using the formula $\int_{f(\gamma)} \omega= \int_{\gamma} f^*\omega$, since $h_Z^{-1}(x_{ij}^\alpha)$ has same homology of $x_{ij}^{\alpha^h}$, where $\alpha^h$ is $h\alpha h ^{-1}$ or $h(\alpha)$ and $x_{ij}^{\alpha^h}$ is the analogue in the corresponding space of $x_{ij}^\alpha$. Hence, we assume that $h_Z=\mathrm{id}$ ($\infty \in Z$). We only prove the case $k<l$ ($\beta \in G$). The easier case $k=l$ can be treated similarly. Assume that $k<l$. The $s$-th component $(x_{ij}^\alpha)_s$ of $x_{ij}^\alpha$ is constant for $s\neq i$ and $(x_{ij}^\alpha)_i=\gamma_i(p_\alpha)$ (see subsection \ref{S23}). Hence:
$$ \int_{x_{ij}^\alpha} \omega_{kl}^\beta=\frac{\delta_{ik}}{2i\pi}\int_{t\in [0,1]}d\log Q_{\beta}(\gamma_i(p_\alpha)(t),p_l)+ \frac{\delta_{il}}{2i\pi}\int_{t\in [0,1]}d\log Q_{\beta}(p_k,\gamma_i(p_\alpha)(t)),$$
where $p_\alpha\in (\Pec \setminus \{\infty\})$ if $i=j$ or $p_\alpha=\alpha \cdot p_j$ if $i< j$, and $Q_\beta$ is equal to $P_{\tilde{\beta}}$ for a given lift $\tilde{\beta}\in \PSL_2(\C)$ of $\beta$. The polynomials $Q_\beta(x,p_l)$ and $Q_\beta(p_k,y)$ are degree $1$ polynomials with respective zeros $x=\beta\cdot p_l$ and $y=\beta^{-1}\cdot p_k$. The equations therefore reduces to:
$$ \int_{x_{ij}^\alpha} \omega_{kl}^\beta=\frac{\delta_{ik}}{2i\pi}\int_{t\in [0,1]}\frac{\gamma_i(p_\alpha)'(t)}{\gamma_i(p_\alpha)(t)-\beta\cdot p_l} + \frac{\delta_{il}}{2i\pi}\int_{t\in [0,1]}\frac{\gamma_i(p_\alpha)'(t)}{\gamma_i(p_\alpha)(t)-\beta^{-1}\cdot p_k},$$
The loops $\gamma_{i}(p_\alpha)$ are oriented clockwise. One deduce from the definition of $\gamma_i(p_\alpha)$ and the residue theorem that: $$ \int_{x_{ij}^\alpha} \omega_{kl}^\beta=\delta_{ik}\delta_{p_\alpha (\beta\cdot p_l)} +\delta_{il}\delta_{p_\alpha (\beta^{-1}\cdot p_k)}. $$
The proposition follows for $k<l$, since $\delta_{p_\alpha (\beta\cdot p_l)}=\delta_{jl}\delta_{\alpha \beta}$ and $\delta_{p_\alpha (\beta^{-1}\cdot p_k)}=0$ if $i=l$ ($i\leq j$ by hypothesis, hence $k<j$).
\end{proof}
\begin{corollary}\label{cor gen H1}
The first singular homology group of $C_n^G(\PP^1\setminus Z)$ with coefficient in $\Z$, is freely generated by the cohomology classes of the loops, $x_{ij}^g$ and $x_{kk}^p$, for $i,j,k\in[1,n], g\in G, p\in Z\setminus \{p_\infty\}$ with $i<j$ .
\end{corollary}
\begin{proof}
We combine the previous proposition with proposition \ref{gen loop}.
\end{proof}
\section{The cohomology ring of the orbit configuration spaces and their homology}\label{S4}
We keep the notation of the previous section $X_n=\Cng$. In the first subsection, we assume that $(G,Z)\neq (\{1\},\emptyset)$. We show using results from the previous section, that for $R \subset \C$ an unital ring, we have isomorphism $H^*(X_n,R)\simeq H^*(X_{n-1},R) \otimes W_n^*$ and $H_*(X_n,R)\simeq H_*(X_n,R)\simeq H^*(X_{n-1},R) \otimes W_n^*$, where $W_n^*\subset \Omega^{1-}_D(X_n)_R$ is a space of differential forms identified to $H^*(F_n,R)$ via integration and $H^*(-,R),H_*(-,R)$ correspond to singular cohomology and homology with coefficients in $R$ respectively. It follows by induction that: the groups $H^k(X_n,R),H_k(X_n,R)$ are free $R$-modules, $H^*(X_n,R)$ admits a basis obtained out of classes corresponding to differential forms, and that the Poincaré series of $X_n$ factors into a product of linear terms. We then show that: integration induces a isomorphism of graded $R$-algebras $\Phi_R:\Omega_D^*(X_n)_R \to H^*(X_n,R)$, that $\Psi_R:A_n(R) \to \Omega_D^*(X_n)_R$ of the previous section is an isomorphism, and that the products (defined in the previous section) spanning $A_n(R)$ and $\Omega_D^*(X_n)_R$ form in fact a basis of the corresponding modules. In particular, we get a description by generators and relations for both rings $H^*(X_n,R)$ and $\Omega_D^*(X_n)_R$. At the end of the section, we prove that $X_n$ is formal is the sense of rational homotopy theory. In the second subsection, we consider the case $(G,Z)=(\{1\},\emptyset)$. We show that the cohomology ring of $X_n$ with coefficient in $R$ correspond to a subring of differential forms (for $R$ a principal ideal domain containing $\frac{1}{2}$, if $n\geq 3$) and that the space $X_n$ is formal. We also give the Poincaré series of $X_n$.\\\\
As mentioned in the introduction in this section: $R\subset \C$ is a unital ring, $H^*(X,R)$ and $H_*(X,R)$ denote the singular cohomology and homology of $X$ with coefficient in $R$ and $X_n:=C_n^G(\PP^1\setminus Z)$.\\\\
For $A^*=\oplus_{k\in \N } A_k, B^*=\oplus_{k\in \N } B_k$ graded $R$-modules and $f^*: A^* \to B^*$ a map of graded $R$-modules, we denote by $A^{ l^-}$ the graded $R$-module $\oplus_{0\leq k \leq l } A_k$ and by $f^{ l^-}$ (resp. $f^l$) the map of graded modules $A^{l^-} \overset{f^{*}}{\to} B^{l^-}$ (resp. $A^l\overset{f^*}{\to} B^l$).
\subsection{The case $(G,Z)\neq (\{1\},\emptyset)$}\label{S41}
In this subsection we assume that $(G,Z)\neq (\{1\},\emptyset)$. Recall that we have an integration isomorphism of graded $\C$-algebras (Cf. Subsection \ref{S DR}), $\int: H_{DR}^*(X_n,\C) {\to} \Hom_\Z(H_*(X_n,\Z),\C)$. Since $\Omega_D^*(X_n)_R$ consists of closed forms, we have a morphism of graded $R$-algebras $T_R: \Omega_D^*(X_n)_R \to H_{DR}^*(X_n,\C), \omega \to [\omega]$. Set:
$$ \phi_*:=\int \circ T_R: \Omega_D^*(X_n)_R \to \Hom_\Z(H_*(X_n,\Z),\C),$$
given by $\phi_*(\omega)([\sigma])= \int_{\sigma^s} \omega$, where $\sigma_s$ is a smooth singular chain representing the homology class $[\sigma]$.
\begin{proposition}\label{prop phi1}
\begin{itemize}
\item[1)] For $\omega \in \Omega_{D}^{1-}(X_n)_R$, the image of $\phi_{1-}(\omega)$ lies in $R$.
\item[2)] We have a well-defined isomorphism $\phi_{1-}^R$ ("restriction" of $\phi_{1-}$) of $R$-modules:
\begin{align*}
\Omega_{D}^{1-}(X_n)_R &\to \Hom_\Z(H_{1-}(X_n,\Z),R)\\
\omega &\mapsto ([\sigma] \mapsto \int_{\sigma^s} \omega),
\end{align*}
\end{itemize}
\end{proposition}
\begin{proof}
The proposition follows from proposition \ref{prop int}, corollary \ref{cor gen H1} and the fact that $X_n$ is path connected.
\end{proof}
We recall that, for $n\geq 1$, we have a fiber bundle $F_{n} \to X_{n} \overset{\pi_n}{\to} X_{n-1}$, with fiber $F_n$ homeomorphic to $\PP^1$ minus $\vert G \vert (n-1)+\vert Z\vert$ points. In particular,
$$H_*(F_n,R)\simeq H^*(F_n,R) \simeq (R)_0\oplus (R^{\alpha_n})_1, $$
where $(-)_i$ stands for the degree $i$ component and $\alpha_n=\vert G \vert (n-1)+\vert Z \vert-1$.
The fiber over a point $P=(p_1,\dots,p_{n-1}) \in X_{n-1}$ is equal to:
$$F_{n,P}=\{(p_1,\dots,p_{n-1},x) \in (\PP^1\setminus Z)^n \vert x\notin G\cdot p_i \: \text{for } i\in[1,n-1]\}.$$
For $X$, a topological space, one has a natural morphism of $R$-modules:
$$ N_{X,*}^R: H^*(X,R) \to \Hom_\Z(H_*(X,\Z),R),$$
The maps $N_{X_n,1-}^R$ and $N_{F_{n,p},*}^R$ are isomorphisms.
\begin{proposition}
Let $W_n^*$ be the $R$-submodule of $\Omega_D^{1-}(X_n)_R$ spanned by $1$ and the $1$-forms $\omega_{in}^g$ and $\omega_{nn}^p$ for $i<n,g\in G$ and $p\in Z\setminus \{p_\infty\}$. For $P\in X_{n-1}$, the composite of the following maps: $$W_n^* \overset{\phi_{1-}^R}{\longrightarrow}\Hom_\Z(H_{1-}(X_n,\Z),R) \overset{(N_{X_n,1-}^R)^{-1}}{\longrightarrow}H^*(X_{n},R) \overset{i_n^*}{\longrightarrow} H^*(F_{n,P},R),$$
where $i_n^*$ is the map induced by the inclusion $i_n$ of the fiber $F_{n,P}\to X_n$ , is an isomorphism of graded $R$-modules.
\end{proposition}
\begin{proof}
The fact that the composition is an isomorphism in degree $0$ is clear. Since $H^*(F_{n,p},R)$ is concentrated in degree $0$ and $1$ we still have to prove the assertion in degree $1$. We have the following commutative diagram:
\[\begin{tikzcd}
W_n^* \arrow{r}{\phi_{1 }^R} \arrow{rd}[swap]{r_1} & \Hom_\Z(H_{1 }(X_n,\Z),R) \arrow{d}{\Hom_\Z((i_n)_*,R) } & \arrow{l}[swap]{ N_{X_n,1 }^R } H^{1 }(X_n,R) \arrow{d}{ i_n^*} \\
& \Hom_\Z(H_1(F_{n,P},\Z),R) & \arrow{l}{N_{F_{n,P},1 }^R} H^1(F_{n,P},R)
\end{tikzcd},\]
where $r_1$ is $\Hom_\Z((i_n)_*,R) \circ \phi_{1 }^R$. Since $ N_{X_n,1 }^R $ and $ N_{F_{n,P},1 }^R $ are isomorphisms, we only need to prove that $r_1$ is an isomorphism.
For any point $Q\in F_{n,p}$, the group $H_1(F_{n,P},R)$ is generated by the homology classes of the loops $x_{in}^\alpha$ and $x_{nn}^\beta$, associated to the base point $Q$, for $i\in[1,n-1],\alpha \in G$ and $\beta \in Z \setminus \{p_\infty\}$. Moreover, $ \int_{x_{in}^\alpha} \omega_{kn}^\beta= \delta_{ik} \delta_{\alpha \beta},$ where $\delta_{ab}=1$ if $a=b$ and $\delta_{a,b}=0$ (proposition \ref{prop int}). This proves that $r_1$ given by $\omega_{kn}^\beta\mapsto ([\gamma]\mapsto \int_{(i_n)_*([\gamma])} \omega_{kn}^\beta)$ is an isomorphism.
\end{proof}
We deduce from the proposition that $W_n^*$ is isomorphic to $H^*(F_{n,p},R)$ and that the restriction of $$\theta_{n,*}:= (N_{X_n,1-}^R)^{-1} \circ \phi_{1- }^R :\Omega_D^{1-}(X_n)_R \to H^*(X_n,R)$$
to $W_n^*$ corresponds to a cohomology extension of the fiber for the fiber bundle $F_n\to X_n \overset{\pi_n}{\to} X_{n-1}$ (in the sense of \cite{Span}, p. 256). Since $H_*(F_n,R)$ is a finitely generated free $R$-module, we can apply the Leray-Hirsch theorem (version in \cite{Span}, p. $259$, theorem $9$) stating in the case of $\pi_n$ that:
\begin{proposition}
We have isomorphisms of $R$-graded modules:
\begin{align*}
H^*(X_{n-1},R)\otimes_R W_{n}^* &\to H^*(X_{n},R)\\
a\otimes b &\mapsto \pi_n^*(a)\smile \theta_{n,*} (b),
\end{align*}
where $\pi_n^*$ is the map induced by $\pi_n:X_n\to X_{n-1}$ and $\smile$ is the cup product; and
\begin{align*}
H_*(X_{n},R)&\to H_*(X_{n-1},R)\otimes_R W_{n}^* \\
a &\mapsto \sum_\alpha (\pi_n)_*(\theta_{n,*}(\omega_\alpha) \frown a)\otimes \omega_\alpha),
\end{align*}
where $(\pi_n)_*$ is induced by $\pi_n$, $\frown$ is the cap product and the sum runs over the elements $$\omega_\alpha\in \{1\}\cup\{\omega_{kn}^c \vert \text{$k\in[1,n],c\in G$ if $k<n$ and $p\in Z \setminus \{p_\infty\}$ if $k=n$}\}.$$
\end{proposition}
Using induction, we deduce the following corollaries:
\begin{corollary}
For $k\geq 0$, the $R$-modules $H^k(X_n,R)$ and $H_k(X_n,R)$ are finitely generated free $R$-modules.
\end{corollary}
\begin{corollary}
The natural map $N_{X_n,*}^R: H^*(X_n,R)\to \Hom_\Z(H_*(X_n,\Z),R)$ is an isomorphism.
\end{corollary}
\begin{corollary}\label{cor base co}
The family of products:
$$\theta_{n,*}(\omega_{i_1j_1}^{\beta_1})\smile \cdots \smile \theta_{n,*}(\omega_{i_kj_k}^{\beta_k}), \quad 0\leq k \leq n , \quad 1\leq i_k\leq j_k \leq n , \quad j_1<j_2<\cdots <j_k $$
with $\beta_k\in G$ if $i_k\neq j_k$ and $\beta_k\in Z\setminus \{p_\infty\}$ if $i_k=j_k$ (for $k=0$ the product is equal to $1$), form a basis of the $R$-module $H^*(X_n,R)$.
\end{corollary}
\begin{corollary}
For $k\geq2$, $P_{X_n}(t)=P_{F_n}(t)P_{X_{n-1}}(t)$, where $P_{X_k}$ and $P_{F_k}$ the poincaré series of $X_k$ and $F_k$ respectively.
\end{corollary}
Since the Poincaré series of $P_{F_k}=(1+\alpha_kt)$, where $\alpha_k=\vert G \vert (k-1) +\vert Z \vert -1$,
\begin{corollary}\label{Po Ser}
For $n\geq 1$:
$$P_{X_n}(t)=\underset{k=1}{\overset{n}{\prod}} (1+\alpha_kt),$$
where $\alpha_k=\vert G \vert (k-1) +\vert Z \vert -1$.
\end{corollary}
\begin{remark}
One can show that the existence of a cohomology extension implies that $\pi_1(X_{n-1})$ acts trivially in the cohomology of the fiber and therefore also acts trivially in the homology of the fiber (since $H_*(F_n,R)$ is free).
\end{remark}
\begin{proposition}\label{prop Phi}
\begin{itemize}
\item[1)] For $\omega \in \Omega_D^*(X_n)_R$, the image of $\phi_*(\omega)$ lies in $R$ and the map $\phi_*$ induces a well-defined morphism of graded $R$-modules:
$$\phi_*^R: \Omega_D^*(X_n)_R \to \Hom_\Z(H_*(X_n,\Z),R),\omega \mapsto ([\sigma] \mapsto \int_{\sigma^s} \omega).$$
\item[2)]The map $\Phi_R:=(N_{X_n,*}^R)^{-1}\circ \phi_{*}^{R}:(\Omega_D^*(X_n)_R,\wedge) \to (H^*(X_n,R),\smile)$ is a morphism of $R$-algebras.
\end{itemize}
\end{proposition}
\begin{proof}
Recall that $H^*(X_n,K)$ for $K$ an abelian group is the cohomology of the cochain complex $\Hom_Z(C_*(X),K)$, where $C_*(X)$ is the singular chain complex of $X$ (over $\Z$). The inclusion $i: R\to \C$ induces a natural map of cochain complexes $i':\Hom_Z(C_*(X),R)\to \Hom_Z(C_*(X),\C)$ compatible with the cup product, and we have a commutative diagram:
\[\begin{tikzcd}
\Hom_\Z(H_*(X_n,\Z),\C) &\arrow{l}[swap]{N_{X_n,*}^\C} H^*(X_n,\C)\\
\Hom_\Z(H_*(X_n,\Z),R) \arrow{u}{\Hom_\Z(H_*(X_n,\Z),i)} &\arrow{l}{N_{X_n,*}^R} \arrow{u}{} \arrow{u}[swap]{H^*(i')} H^*(X_n,R)
\end{tikzcd} \]
where $H^*(i')$ is a morphism of algebras. In particular, if $[\sigma_i]_\C \in H^*(X_n,\C)$ (for $i\in [1,k]$) is equal to $H^*(i')[\sigma_i]_R$ for $[\sigma_i]_R\in H^*(X_n,R)$, then:
$$ N_{X_n,*}^\C([\sigma_1]_\C \smile \cdots \smile [\sigma_k]_\C)= \Hom_\Z(H_*(X_n,\Z),i)(N_{X_n,*}^R)^{-1}([\sigma_1]_R \smile \cdots \smile [\sigma_k]_R).$$
Recall that $\phi_*=\int \circ T_R$. Hence, $(N_{X,*}^\C)^{-1}\circ \phi^*=(N_{X,*}^\C)^{-1} \circ \int \circ T_R$ is a map of algebras. Indeed, $(N_{X,*}^\C)^{-1} \circ \int=\int'$ (Cf. Subsection \ref{S DR}) and $T_R$ are both morphisms of algebras. From this and the last equation containing cohomology classes, we deduce (since $\phi_1^R$ is the "restriction" of $\phi_*$) that:
$$ \phi_*(\omega_1\wedge \cdots \wedge \omega_k)=\Hom_\Z(H_*(X_n,\Z),i)(N_{X_n,*}^R)^{-1}( \Phi_{1,R}(\omega_1) \smile \cdots \smile \Phi_{1,R}(\omega_k) ),$$
for $\omega_1,\dots,\omega_k \in \Omega_D^1(X_n)_R$ and where $\Phi_{1,R }:= (N_{X_n,1}^R)^{-1}\phi_1^R$. In particular, the image of $\phi_*(\omega_1\wedge \cdots \wedge \omega_k)$, with $\omega_i \in \Omega_D^1(X_n)_R$, lies in $R$ and hence $\phi_*$ induces a well-defined integration morphism $\phi_*^R:\Omega_D^*(X_n)_R\to \Hom_\Z(H_*(X_n,\Z),R), \omega \mapsto ([\sigma] \mapsto \int_{\sigma^s} \omega)$. This proves point (1) of the proposition. Point (2) also follows from the last equation. We have proved the proposition.
\end{proof}

\begin{proposition}
\begin{itemize}
\item[1)] The maps $\Psi_R: A_n(R) \to \Omega_D^*(X_n)_R$ and $\Phi_R: \Omega_D^*(X_n)_R \to H^*(X_n,R)$, of propositions \ref{prop Psi} and \ref{prop Phi}, are isomorphisms of graded $R$-algebras.
\item[2)] The family of products in corollary \ref{cor base} and there analogues for $A_n(R)$ form basis of $\Omega_D^*(X_n)_R$ and $A_n(R)$ respectively.
\end{itemize}
\end{proposition}
\begin{proof}
The proposition follows from the fact that $\Phi_R$ and $\Psi_R$ are algebra morphisms (propositions \ref{prop Psi} and \ref{prop Phi}), corollary \ref{cor base}, proposition \ref{prop Psi} and corollary \ref{cor base co}.
\end{proof}
Denote by $\Omega^*(X_n)_\C$ the commutative cochain algebra of complex valued differential forms on $X_n$, with differential the exterior differential and product the wedge product.
\begin{corollary}
The commutative cochain algebras $\Omega_D^*(X_n)_\C$, $H^*(X_n,\C)$ and $\Omega^*(X_n)_\C$, where the first two algebras are equipped with a zero differential, are weakly equivalent.
\end{corollary}
\begin{proof}
The previous proposition states that $\Omega_D^*(X_n)_\C$ and $H^*(X_n,\C)$ are isomorphic. Since the forms in $\Omega_D^*(X_n)_\C$ are closed the inclusion $\tilde{T}_\C:\Omega_D^*(X_n)_\C \to \Omega^*(X_n)_\C $ is a map of cochain algebras. Since the map induced by $\tilde{T}_\C$ in cohomology is the map $T_\C$ and corresponds to $(\int')^{-1}\circ \Phi_C$, where $\int'$ is the De Rham isomorphism, we deduce that $T_\C$ is an isomorphism and $\tilde{T}_\C$ is a quasi-isomorphism.
\end{proof}
\begin{corollary}
The space $X_n$ is formal in the sense of rational homotopy theory, i.e. the commutative cochain algebra $H^*(X_n,\Q)$ with zero differential is weakly equivalent to the cochain algebra $A_{PL}^*(X_n)_\Q$ of polynomial differential forms on $X$ with coefficient in $\Q$.
\end{corollary}
\begin{proof}
Since the rational cohomology of $X_n$ is of finite type, the formality of $A_{PL}^*(X_n)_F$ for any field $F$ containing $\Q$ for $X_n$ will imply the formality for $\Q$ and $A_{PL}^*(X_n)_\C$ will be weakly equivalent to $\Omega^*(X_n)_\C$ (Cf. Subsection \ref{S formality}). Hence, the above corollary becomes a consequence of the previous one.
\end{proof}
Recall that $\Gamma:=G^n\rtimes \mathfrak{S}_n$ acts on $C_n^G(\PP^1\setminus Z)$. The quotient map $\pi: X_n\to X_n/\Gamma$ is a covering map and hence $\pi$ induces a cdga isomorphism $\pi^*:\Omega^*(X_n/\Gamma)_\C \to \Omega^*(X_n)_\C^{\Gamma}$ between complexes of complex valued forms (the superscript $\Gamma$ is (and will be) used for $\Gamma$ invariants). We know that $\Omega_D^*(X_n)_\C$ is stable under $\Gamma$ (Cf. Section \ref{S31}) and that the inclusion $\Omega_D^*(X_n)_\C\to \Omega^*(X_n)_\C$ is a quasi-isomorphism. It follows that the last inclusion also induces a quasi-isomorphism $\Omega_D^*(X_n)_\C^\Gamma \to \Omega^*(X_n)_\C^{\Gamma}$. In particular, $\Omega^*(X_n/\Gamma)_\C$ is quasi-isomorphic to its cohomology (and $\Omega_D^*(X_n)_\C^\Gamma$). One can derive from this that $X_n/ \Gamma$ is formal.

\subsection{The case $(G,Z)=(1,\emptyset)$, i.e. $X_n=C_n(\PP^1)$}\label{S42}
Let $R \subset \C$ be an unital ring. The space $\PP^1$ is homeomorphic to $S^2 \subset \R^3$. We will sometimes switch spaces implicitly for convenience. One can derive the formality of the upcoming spaces by theoretic arguments. Here, we derive this from the description of the cohomology using differential forms, since we are interested in the description also.\\\\
By definition $C_1(\PP^1)$ is equal to $\PP^1$. The singular cohomology algebra $H^*(\PP^1,R)$ of $\PP^1$ is isomorphic to the $R$-subalgebra of differential forms generated by the volume form $\omega=\frac{i}{2\pi}\frac{dz\wedge d\bar{z}}{(1+ \vert z \vert^2)^2}$ (the integral of $\omega$ over $\PP^1$ with its canonical orientation is equal to $1$).\\\\
We consider the case $n=2$. It is known that the projection $\pi_2:C_2(S^2) \to C_1(S^2)$ is a homotopy equivalence with homotopy inverse $s$ given by $s(p)=(p,-p)$. Indeed, $\pi_2 \circ s= \mathrm{id}$, and $$H((x_1,x_2),t)=(x_1, \frac{(1-t)x_2-tx_1}{ \vert\vert (1-t)x_2-tx_1\vert \vert }),$$ where $\vert \vert \cdot \vert \vert$ is the Euclidean norm in $\R^3$, is a homotopy between $s\circ \pi_2$ and $\mathrm{id}$. Hence, $H^*(C_2(\PP^1),R)$ correspond to the $R$-subalgebra of differential forms generated by $\pi_2^* \omega$ and $C_2(\PP^1)$ is formal . \\\\
We examine $C_n(\PP^1)$, for $n\geq 3$. We assume that $R$ is a principle ideal domain. For $\underline{p}=(p_1,p_2,p_3)\in (\PP^1)^3$ where the coordinates are pairwise distinct (i.e. a projective basis of $\PP^1$), there exists a unique $h_{\underline{p}}\in \PGL(\C^2)$ mapping $p_1,p_2,p_3$ to $0,1, \infty$ respectively. The homography $h_{\underline{p}}$ is given by:
$$h_{\underline{p}}(z)=\frac{(z-p_1)(p_2-p_3)}{(p_2-p_1)(z-p_3)}.$$
This proves that $\PGL(\C^2)$ acts freely transitively on $C_3(\PP^1)$. Hence, $C_3(\PP^1)$ is diffeomorphic to $\PGL(\C^2)$ and for $n\geq 3$, we have an homeomorphism:
\begin{align}\label{Split Fib}\begin{split}
C_n(\PP^1)&\to C_3(\PP^1)\times C_{n-3}(\PP^1\setminus \{0,1,\infty\})\\ (p_1,\dots, p_n) &
\mapsto ((p_1,p_2,p_3), (h_{\underline{p}}(p_4),\dots,h_{\underline{p}}(p_n)))\end{split},\end{align}
where ${\underline{p}}=(p_1,p_2,p_3)$.
This decomposition was used in \cite{ziegC} to compute the cohomology of $C_n(\PP^1)$ for $n\geq 3$.\\\\ Using the $QR$ decomposition one obtains that $C_3(\PP^1)\simeq \PGL(\C^2)$ is diffeomorphic to $\PSU(\C^2)\times \R^3$. We have $\PSU(\C^2) \simeq \mathrm{SO}(\R^3)$ and $\mathrm{SO}(\R^3)$ is homeomorphic to the three-dimensional projective real space $\PP^3(\R)$. In particular, $C_3(\PP^1)$ is homotopy equivalent to $\PP^3(\R)$ (see also remark \ref{rmk steif} section \ref{S5}) and for $\frac{1}{2} \notin R$, the ring $H^*(C_n(\PP^1),R)$ can not be described as an algebra of closed differential forms, since it contains torsion. Now take a volume form $\omega_V$ on $\PSU(\C^2)$ for which $\PSU(\C^2)$ have volume $1$ (one can normalize any volume form) and denote by $\omega'_V$ the pullback of $\omega_V$ by the smooth map $C_3(\PP^1)\simeq \PGL(\C_2)\simeq \PSU(\C^2)\times \R^3\to \PSU(\C^2)$, where the last map is the projection. Using the cross product one can check that, for $R$ containing $\frac{1}{2}$, the ring $H^*(C_n(\PP^1),R))$ is isomorphic via integration to the subalgebra of closed differential forms on $C_3(\PP^1)\times C_{n-3}(\PP^1\setminus \{0,1,\infty\})$ given by $\Omega_{3,R}^*\otimes \Omega_D^*(C_n(\PP^1\setminus \{0,1,\infty\})$ (tensor with Koszul sign convention), where $\Omega_{3,R}=(R)_0\oplus (R\omega_V')_3$ (indices for the degree) is the subalgebra generated by $\omega_V'$. This gives a description of $H^*(C_n(\PP^1),R)$ as an algebra of closed differential forms (for $\frac{1}{2}\in R$) via integration and proves that $C_n(\PP^1)$ is formal.\\\\
We have shown that $C_n(\PP^1)$ is formal for $n\geq 1$, and described the singular cohomology ring of $C_n(\PP^1)$ with coefficient in an unital ring $R\subset \C$, using differential forms with constraints on $R$ for $n\geq 3$. One can easily derive presentations of the cohomology ring from what we have seen above and (for $n\geq 3$) the presentation of the cohomology ring of $C_n(\PP^1\setminus \{0,1,\infty\})$ from the previous section (see also \cite{ziegC}).
\begin{proposition}\label{cor poinc S2}
The Poincaré series $P_n$ of $C_n(\PP^1)$ is given by:
$$P_1(t)=P_2(t)=t^2, \quad \text{and} \quad P_n(t)=t^3\prod_{k=1}^{n-3} (1+\beta_kt), $$
for $n\geq$ and $\beta_k=1+k$.
\end{proposition}
\begin{proof}
$C_1(\PP^1)$ is a $2$-sphere. Hence, $P_1(t)=t^2$. The space $C_2(\PP^1)$ is homotopy equivalent to $\PP^1$ as seen previously in this subsection and hence $P_2=P_1$. For $n\geq 3$, $C_n(\PP^1) \simeq C_3(\PP^1) \times C_{n-3}(\PP^1\setminus \{0,1,\infty\})$. Therefore, for $n\geq 3$, $P_n$ is the product of the Poincaré series of $C_3(\PP^1)$ and $C_{n-3}(\PP^1\setminus \{0,1,\infty\})$. We have seen in this subsection that $C_3(\PP^1)$ is homotopy equivalent to $\PP^3(\R)$ and hence its Poincaré series is given by $t^3$. The Poincaré series of $C_{n-3}(\PP^1\setminus \{0,1,\infty\})$ is equal to the other factor in the formula above by corollary \ref{Po Ser}.
\end{proof}
\section{Homotopy groups and LSC formula}\label{S5}
In the first subsection, we consider the case $(G,Z)\neq (\{1\},\emptyset)$. We prove that $X_n:=C_n^G(\PP^1\setminus Z)$ is a $K(\pi,1)$ and that the fibration $X_{n+1} \to X_n$ admits a cross section. The cross section is used to prove that the first homotopy group $\pi_1X_n$ of $X_n$ is an iterated almost direct product of free groups (Cf. Section \ref{S IA} for the definition). This gives an LCS formula relating the Poincaré series of $X_n$ to the rank of quotients of successive terms of the lower central series of $\pi_1X_n$. In the second subsection, we study the remaining case $X_n=C_n(\PP^1)$. We give the higher homotopy groups of $X_n$ known in the literature, describe the structure of the fundamental group of $X_n$ (the structure is also known, in fact a presentation is known) and show that we have an LCS formula.\\\\
Let $H$ be a group. We denote by $\{\Gamma_i H\}_{i\geq1}$ the lower central series of $H$ and by $H^{ab}$ the abelianization of $H$ (i.e. $H/\Gamma_2 H$).
\subsection{The case $(G,Z)\neq (\{1\}, \emptyset)$}
Till the end of this subsection we assume that $(G,Z)\neq (\{1\}, \emptyset)$.
\begin{proposition}
The space $X_n$ is aspheric, i.e. a $K(\pi,1)$ space.
\end{proposition}
\begin{proof}
The fiber $F_{n+1}$ of $X_{n+1}\to X_n$ is a $2$-sphere with finite number (nonzero) of punctures and hence aspheric. Using this, we deduce from the long exact sequence of the fibration the following exact sequences:
$$ 0\to \pi_k(X_{n+1})\to \pi_k( X_n) \to 0\quad \text{for $k\geq 3$} \qquad \text{and} \quad 0\to \pi_2(X_{n+1})\to \pi_2( X_n) \to \pi_{1} F_{n+1} .$$
Hence, $\pi_k(X_{n+1})=\pi_k(X_1)=\pi_k(\PP^1\setminus Z)$ for $k\geq 3$, and that $\pi_2(X_{n+1})$ injects into $\pi_2(X_1)=\pi_2(\PP^1 \setminus Z)$. This shows that $X_n$ is aspheric, since $\PP^1\setminus Z$ is aspheric.
\end{proof}

We now study the fundamental group of $X_n$. Since $F_{n+1},X_{n+1}$ and $X_n$ are $K(\pi,1)$ spaces and are path connected. The long exact sequence of the fibration $X_{n+1}\to X_n$ gives the exact sequence:
\begin{equation}\label{Suite ex}
1\to \pi_1(F_{P,n+1},\tilde{P})\to \pi_1(X_{n+1},\tilde{P})\to \pi_1(X_n,P)\to 1,
\end{equation}
where $\tilde{P} \in X_{n+1}$ is a preimage of $P\in X_n$, $F_{P,n+1}$ is the fiber over $P$ and the maps are induced by the inclusion $F_{P,n+1} \to X_{n+1}$ and the projection $X_{n+1}\to X_n$.
\begin{proposition}
The fibration $X_{n+1}\to X_{n}$ admits a continuous cross-section $s_n:X_n \to X_{n+1}$.
\end{proposition}
\begin{proof}
Fix a $p_\infty \in Z$ and a Riemannian metric on $\PP^1$. Let $d$ be the distance function associated to the Riemannian metric. One can find a real number $\eta>0$ and an isometry $f: [0,\eta]\to \PP^1$ such that $\alpha(0)=p_\infty$ and $\alpha(]0,\eta])\subset \PP^1\setminus Z$. For $P=(p_1,\dots,p_n)\in X_n$, set:
$$C_{P}=\frac{1}{2}\underset{(i,g)\in [1,n]\times G}{\mathrm{Min}} d(p_\infty,g\cdot p_i) \qquad \text{and} \qquad C_{P}^\eta= \mathrm{Min}(C_{P},\eta).$$
For $(i,g)\in [1,n] \times G$, we have $d(p_\infty,\alpha(C_{P}^\eta))< d(p_\infty, g\cdot p_i)$. Therefore, the mapping $X_{n}\to X_{n+1}, (p_1,\dots,p_n)\mapsto(p_1,\dots,p_n, \alpha(C_{P}^\eta))$ defines a continuous cross section.
\end{proof}
\begin{corollary}\label{Suite spl}
The exact sequence (\ref{Suite ex}) splits and $\pi_k(X_{n+1},s_n(P))$ is isomorphic to $$\pi_1(F_{P,n+1},s_n(P)) \rtimes_{\beta_n} \pi_1(X_n,P),$$ where $\beta_n([a])([b])=[s_n(a)bs_n(a)^{-1}]$ for $a,b$ loops of $X_n$ and $F_{P,n+1}$ based at $P$ and $s_n(P)$ respectively.
\end{corollary}
Let $1\to N\to H \to K\to 1$ be an exact sequence of groups and $s:K\to H$ a section (map of sets). We have a morphism $\alpha_s:K\to \Aut(N)$ mapping $a\in K$ to the conjugacy by $s(a)$ in $N$. The map $\alpha_s$ induces a morphism $\beta: K \to \Aut(N^{ab})$, independent of the choice of $s$.
\begin{lemma}
Let $N,H,K$ and $\beta$ be as in the previous paragraph. If the natural map $N^{ab}\to H^{ab}$ is injective then the image of $\beta$ consists of $IA$-automorphisms.
\end{lemma}
\begin{proof}
For $L$ a group, denote by $(x,y)$ the commutator of $x,y \in L$ and by $(L,L)$ the commutator subgroup of $L$. The injectiveness condition means that $N\cap (H,H)=(N,N)$. Therefore, under the injectiveness assumption, for $c\in K$: $s(c)bs(c)^{-1}b^{-1} \in N\cap (H,H)=(N,N)$ ($s:K\to H$ is the set section), and $\beta(c)(b)$ and $b$ are equal in $N^{ab}$. We have proved the proposition.
\end{proof}
\begin{lemma}
For $P\in X_n$, the inclusion of the fiber $F_{P,n+1}$ over $P$ of $X_{n+1} \to X_n$ into $X_{n+1}$, induces an injective morphism in homology.
\end{lemma}
\begin{proof}
We can use corollary \ref{cor gen H1} and the fact that $H_1(F_{P,n+1},\Z)$ is generated by the loops $x_{in+1}^g, x_{n+1n+1}^p$ (with base point a preimage of $P$) for $i\in[1,n],g\in G$ and $p\in Z\setminus\{p_\infty\}$ (for any $p_\infty \in Z$).
\end{proof}
\begin{proposition}
The semidirect product $\pi_1(F_{P,n+1},s_n(P)) \rtimes_{\beta_n} \pi_1(X_n,P)$ of corollary \ref{Suite spl} is an almost direct product, i.e. the image of $\beta_n$ consists of $IA$-automorphisms.
\end{proposition}
\begin{proof}
The proposition follows from the last two lemmas.
\end{proof}
As we have already seen $F_{P,n+1}$ is a $2$-sphere minus $\alpha_{n+1}+1=\vert G \vert n +\vert Z \vert $ points. Hence, $\pi_1(F_{P,n+1},s_n(P))$ is isomorphic to the free group $F(\alpha_{n+1} )$ on $\alpha_{n+1}$ generators.
\begin{corollary}\label{cor Fun LCS}
\begin{itemize}
\item[1)]The fundamental group of $X_n$ is isomorphic to an iterated almost direct product of free groups:
$$F(\alpha_n)\rtimes_{\gamma_n} (F(\alpha_{n-1})\rtimes_{\gamma_{n-1}} (F(\alpha_{n-2})\rtimes_{\gamma_{n-1}} (\cdots \rtimes_{\gamma_{3}} (F(\alpha_2) \rtimes_{\gamma_{2}} F(\alpha_1))\cdots ))),$$
where $\alpha_k=\vert G \vert (n -1)+\vert Z \vert -1$, for $k\in [1,n]$.

\item[2)] Let $P_{X_n}$ is the Poincaré series of $X_n$:
$$P_{X_n}(t)=\underset{k=1}{\overset{n}{\prod}} (1+\alpha_kt)=\prod_{i\geq 1} (1-t^i)^{\phi_i(\pi_1X_n)},$$
where $\alpha_k=\vert G \vert (n -1)+\vert Z \vert -1$, $\phi_i(\pi_1X_n)$ is the rank of the quotient $\Gamma_i \pi_1X_n/\Gamma_{i+1} \pi_1 X_n$, with $\Gamma_l \pi_1 X_n$ the $l$-th term of the lower central series of $\pi_1 X_n$ and $\phi_i(\pi_1X_n)$ is equal to:
$$ \phi_i(\pi_1X_n)=\sum_{k=1}^n \phi_i(F(\alpha_k))=\frac{1}{i}\sum_{k=1}^n \sum_{j\vert i} \mu(j)\alpha_k^{i/j}, $$
with $\mu$ the Möbius function.
\end{itemize}
\end{corollary}
\begin{proof}
Point $(1)$ is obtained by induction on $n$ using the previous proposition and the fact that the fundamental group of the fiber $F_k$ of $X_{k+1} \to X_k$ is $F(\alpha_k)$. We have already proved that the Poincaré series of $X_n$ is given by $P_{X_n}(t)=\prod_{k=1}^n (1+\alpha_kt)$ (corollary \ref{Po Ser}). By applying the LCS formula of proposition \ref{LCS Formula} and equation (\ref{eq mob}) (Subsection \ref{S IA}), we get the formula relating $P_{X_n}$ to the infinite product and the formula for the constants $\phi_i(\pi_1X_n)$.
\end{proof}
\begin{remark}\label{rmk Malcev}
We have introduced in \cite{MM} a graded Lie algebra $\mathfrak{p}_n(G)(\Q)$ whose degree completion correspond to the Malcev Lie algebra (\cite{Q1}) $\mathrm{Lie}_\Q\pi_1X_n$ of $\pi_1X_n$ over $\Q$, for $G\neq \{1\}$ and $Z$ equal to the set of irregular points of $G$ (the general case is treated in \cite{MMT}). Since the associated graded of $\mathrm{Lie}_\Q\pi_1X_n$ is isomorphic to $(\oplus_{i\geq 1}\Gamma_i \pi_1X_n/\Gamma_{i+1} \pi_1 X_n) \otimes \Q$ the constant $\phi_i(\pi_1X_n)$ above correspond to the dimension of the degree $i$ part of $\mathfrak{p}_n(G)(\Q)$ (the same holds for the Lie algebras in \cite{MMT}, see also proposition \ref{LCS S2} of the next subsection).
\end{remark}
\begin{remark}
Knowing that $\pi_1(X_n)$ is an iterated almost direct product of free groups, one can deduce the Poincaré series of $X_n$ from the work of \cite{CohSuc} (Cf. Remark \ref{rmk POP}, Subsection \ref{S IA}), since $X_n$ is a $K(\pi,1)$.
\end{remark}
\begin{remark}\label{rmk res}
An iterated almost direct product of free groups is residually torsion free nilpotent (\cite{BBR}, Cf. end of Section $2$). In particular, $\pi_1X_n$ is residually torsion free nilpotent. The action of $G^n$ on $X_n$ is a covering action and $X_n/G^n\simeq C_n(\C\setminus W)$ where $W$ is a finite set. We can show that $\pi_1C_n(\C \setminus W)$ contains the pure braid group $\pi_1C_n(\C)$ as a subgroup. The pure braid group $\pi_1C_n(\C)$ is not residually free for $n\geq 4$ (\cite{Rfree}) and hence $\pi_1 X_n$ is not residually free.
\end{remark}

\subsection{The case $(G,Z)=(\{1\},\emptyset)$}
In this subsection we consider the case $X_n=C_n(\PP^1)$. We recall (Cf. Subsection \ref{S32}) that $C_2(\PP^2)$ is homotopy equivalent to $C_1(\PP^1) =\PP^1\simeq S^2$, $C_3(\PP^1)$ is homotopy equivalent to the $3$-dimensional real projective plane $\PP^3(\R)$, and that $C_n(\PP^1)$ is homeomorphic to $C_3(\PP^1)\times C_n(\PP^1\setminus \{0,1,\infty\})$ for $n\geq 3$.
\begin{proposition}
The higher homotopy groups of $C_n(\PP^1)$ are given by,
$$ \pi_2C_n(\PP^1)=
\begin{cases}
\Z,& \text{if } n\leq 2\\
0, & \text{otherwise}
\end{cases} \quad \text{and} \quad \pi_kC_n(\PP^1)= \pi_kS^2 ,$$
for $k\geq 3$. The space $C_n(\PP^1)$ is simply connected if $n\leq 2$, $\pi_1C_3(\PP^1)\simeq \Z/2\Z$ and for $n> 3$, $\pi_1C_n(\PP^1)\simeq \Z/2\Z\times \pi_1C_{n-3}(\PP^1\setminus \{0,1,\infty\})$.
\end{proposition}
\begin{proof}
The proposition follows from the facts reminded before it, and the fact that $C_{n-3}(\PP^1\setminus \{0,1,\infty\})$ is aspheric as we have seen in the previous subsection.
\end{proof}
\begin{remark}\label{rmk steif}
A homotopy equivalence between $C_3(\PP^1)$ and the Stiefel manifold $V_2(\R^3)\simeq \mathrm{SO}(\R^3)\simeq \PP^3(\R)$ appears in \cite{FadVan}, and the higher homotopy groups of $C_n(S^r)$ ($S^r$ the $r$-sphere) for $r\geq 2$ were computed in in term of homotopy groups of Stiefel manifolds of orthogonal frames in \cite{Fadell}.
\end{remark}
\begin{proposition}\label{LCS S2}
\begin{itemize}
\item[1)] For $n>3$, $\pi_1C_n(\PP^1)$ is the product of $\Z/2\Z$ with the iterated almost direct product:
$$F(\beta_{n-3})\rtimes_{\varepsilon_{n-3}} (F(\beta_{n-2})\rtimes_{\varepsilon_{n-2}} (F(\beta_{n-3})\rtimes_{\varepsilon_{n-3}} (\cdots \rtimes_{\varepsilon_{3}} (F(\beta_2) \rtimes_{\varepsilon_{2}} F(\beta_1))\cdots ))),$$
where $\beta_k=k+1$.
\item[2)] Denote by $P_n$ the Poincaré series of $C_n(\PP^1)$. For $n>3$:
$$ \frac{1}{t^3}P_{n}(t)=\underset{k=1}{\overset{n-3}{\prod}} (1+\beta_kt)=\prod_{i\geq 1} (1-t^i)^{\phi_i(\pi_1C_n(\PP^1))},$$
where $\beta_k$ is as in $(1)$, $\phi_i(\pi_1C_n(\PP^1))$ is the rank of the quotient $\Gamma_i \pi_1C_n(\PP^1)/\Gamma_{i+1} \pi_1 C_n(\PP^1)$, with $\Gamma_l \pi_1 C_n(\PP^1)$ the $l$-th term of the lower central series of $\pi_1 C_n(\PP^1)$ and
$$ \phi_i(\pi_1C_n(\PP^1))=\sum_{k=1}^n \phi_i(F(\beta_k))=\frac{1}{i}\sum_{k=1}^n \sum_{j\vert i} \mu(j)\beta_k^{i/j}, $$
with $\mu$ the Möbius function.
\end{itemize}
\end{proposition}
\begin{proof}
Point $(1)$ follows from the isomorphism concerning $\pi_1C_n(\PP^1)$ in the previous proposition and $(1)$ of corollary \ref{cor Fun LCS} giving the structure of $\pi_1C_n(\PP^1\setminus \{0,1,\infty\})$. We prove $(2)$, $\pi_1C_n(\PP^1)\simeq \Z/2\Z \times K$, where $K$ is the iterated almost direct product in the proposition. Since the factor $\Z/2\Z$ is central in $\pi_1C_n(\PP^1)$, $\Gamma_l \pi_1C_n(\PP^1)=\Gamma_l K$. Therefore, the equation $\prod_{k=1}^{n-3} (1+\beta_kt)=\prod_{i\geq 1} (1-t^i)^{\phi_i(\pi_1C_n(\PP^1))}$ and the formula for $\phi_i(\pi_1C_n(\PP^1))$ can be obtained by applying the LCS formula (Cf. Proposition \ref{LCS Formula}) and equation (\ref{eq mob}) of subsection \ref{S IA} to $K$. Finally, $ \frac{1}{t^3}P_{n}(t)=\prod_{k=1}^{n-3} (1+\beta_kt)$ by proposition \ref{cor poinc S2} of subsection \ref{S42}.
\end{proof}
\begin{remark}
The iterated almost direct product in the previous proposition is residually nilpotent (Cf. Remark \ref{rmk res}) and hence $\pi_1C_n(\PP^1)$ is residually nilpotent. The group $\pi_1 C_n(\PP^1)$ is not residually free, since it contains torsion. The iterated almost product factor is not residually free for $n\geq 7$ as we have seen in remark \ref{rmk res}.
\end{remark}
\section{Appendix: finite actions on surfaces with punctures}
Let $S$ be a compact boundaryless, oriented surface and $Y$ a finite subset (eventually empty) of $S$. In this appendix, we show that if a finite group $H$ acts by orientation preserving homeomorphisms on $S\setminus Y$, then the action is the restriction of an action of $H$ on $S$ stabilizing $Y$ and that $S$ admits a complex structure in which $H$ acts holomorphically. We have applied this result in section \ref{S1} for the case $S=S^2$. The result seems to be known, but the author was not able to track a proof in the literature.\\\\
We implicitly assume that a surface is second-countable, Hausdorff and connected. A closed surface is a compact surface with no boundary. By a disc we mean a topological disk.
\begin{proposition}\label{prop disk}
Let $H$ be a finite group acting by orientation preserving homeomorphisms on an oriented connected surface $S$. For $p$ in the interior of $S$ (not on the boundary), and any neighborhood $U_p$ of $p$, there exists a closed disc $D_p\subset U_p$ containing $p$ in its interior such that:
\begin{itemize}
\item[1)] The stabilizer $H_p$ of $p$ stabilizes $D_p$.
\item[2)] The action of $H_p$ on $p$ is equivalent to the action of the cyclic subgroup of order $\vert H_p \vert$ of $\mathrm{SO}(\R^2)$ on the closed unit disk in $\R^2$. In particular, $H_p$ is cyclic.
\item[3)] The intersection $h\cdot D_p \cap D_p$, for $h\in H$, is non-empty if and only if $h\in H_p$.
\item[4)] The set of points with a non-trivial stabilizer in $D_p$ is either empty, either reduced to $p$.
\end{itemize}
\end{proposition}
\begin{proof}
In this proof, we mean by "a disk $D$ around $p$" a closed disk $D$ containing $p$ in its interior $\mathring{D}$. Note that $(3)$ and $(2)$ imply $(4)$ and that the condition $D_p \subset U_p$ follows from $(2)$, since $(2)$ implies that the disc can be taken as small as needed. Therefore, we will only prove the existence of a disk satisfying $(1)$, $(2)$ and $(3)$. We first show that a disk around $p$ satisfying $(3)$ exists. Take a neighborhood $V_q$ for each $q\in H\cdot p$, such that $V_q$ and $V_{q'}$ are disjoint for $q\neq q'$, and set $V'_p=\cap_{h\in H} h^{-1} V_{h\cdot p}$. The set $V'_p$ is a neighborhood of $p$ such that $h'\cdot V'_p \subset V_{h'\cdot p}$, for $h'\in H$. Hence, any disk around $p$ contained in $V'_p$ satisfies $(3)$. Let $D$ be such a disk and $D'$ a closed disk (around $p$) lying in the non-empty interior of $\cap_{h\in H_p} h\cdot D \subset D$. For $h\in H_p$, $h\cdot D'$ lies $\mathring{D}$. Since all the disks $h\cdot D'$ lie in $\mathring{D}$, it follows form a general argument on intersection of disks in $\R^2$ (\cite{const}, proposition 2.4) that the closure $J_p$ of the connected component of $\cap_{h\in H_p} h\cdot \mathring{D}'$ containing $p$ is a closed disk. Clearly $J_p$ satisfies $(1)$ and $(3)$. We still have to prove that $(2)$ is satisfied for $J_p$ and $H_p$. An orientation preserving homeomorphism $f$ of finite order $k$ of the unit closed disc $D"\subset \R^2$ is conjugate to a rotation of $D"$ (\cite{const}, theorem 3.1). Hence, it suffices to prove that $H_p$ is cylic. By the same conjugacy result $H_p$ acts freely on the boundary $\partial J_p\simeq S^1$ of $J_p$. In particular, $J_p\to (\partial J_p) /H_p$ is a normal cover and $H_p$ is a quotient of $\pi_1((\partial J_p)/H_p)$. The quotient $(\partial J_p)/H_p) $ is a circle, since it is a compact boundaryless $1$-dimensional manifold. Hence, $H_p$ is a quotient of $\Z$ and $H_p$ is cyclic. We have proved the proposition.
\end{proof}

\begin{corollary}\label{cor fix}
The set of points in $S$ having a non-trivial stabilizer is finite if $S$ is closed.
\end{corollary}
\begin{proof}
Denote by $X$ the set in the corollary. Point $(4)$ of the last proposition implies that $X$ is discrete. The corollary follows, since $X$ is closed. We prove that $X$ is closed. For $h\in H$ denote by $f_h$ the map $S\to S\times S,p\mapsto(p,h\cdot p)$. The diagonal $\Delta_S\subset S\times S$, is closed and $X=\cup_{h\in H} f_{h}^{-1}(\Delta_S)$. Hence, $X$ is closed.
\end{proof}

\begin{proposition}\label{prop quot}
Let $H$ be a finite group acting by orientation preserving homeomorphisms on a connected closed oriented surface. Denote by $Y$ the set of points having a non-trivial stabilizer. The quotient space $S/H$ is a surface and the quotient map $(S\setminus Y)\to (S\setminus Y)/H$ is a normal cover.
\end{proposition}
\begin{proof}
We will use the notations of the previous proposition. The quotient map $\pi: S\to S/H$ is open. The space $S/H$ is Hausdorff, since $S$ is Hausdorff and $H$ is finite. Point $(2)$ of the previous proposition implies that $(S\setminus Y)\to (S\setminus Y)/H$ is a (normal) cover and that $\pi(D_p) \simeq D_p/H_p$. By $(3)$ of the previous proposition $D_p/H_p$ is a disk this proves that the neighborhood $\pi(D_p)$ ($\pi$ is open) of $\pi(p)$ is Euclidean. We have proved the proposition.
\end{proof}
\begin{lemma}
Let $H,S$ and $Y$ be as in the previous proposition. Denote by $\pi$ the quotient map $S\to S/H$. The surfaces $S$ and $S/H$ admit complex structures (Riemann surface structures) such that $\pi$ is holomorphic.
\end{lemma}
\begin{proof}
The quotient space $S/H$ is a closed surface and hence admits a complex structure. Such a structure can be lifted to $S$ via $\pi$ (Cf. \cite{Fult} §19.B) and $\pi$ becomes holomorphic.
\end{proof}
\begin{proposition}\label{prop Com Str}
If $H$ is a finite group acting by orientation preserving homeomorphisms on a closed surface $S$ then there is a complex structure on $S$ in which $H$ acts holomorphically.
\end{proposition}
\begin{proof}
By the previous lemma, we have complex structures over $S$ and $S/H$ such that $\pi:S\to S/H$ is holomorphic. Take $p\in S\setminus Y$ an $U_p$ a neighborhood of $p$ such that $\pi_{\vert U_p}: U_p \to \pi({U_p})$ is an homeomorphism ($(S\setminus Y) \to (S\setminus Y)/H$ is a cover by proposition \ref{prop quot}). The map $\pi_{\vert h\cdot U_p}: h\cdot U_p \to \pi(h\cdot U_p)=\pi(U_p)$ is also a homeomorphism and $U_p \to h\cdot U_p$, induce by $h\in H $ is equal to $\pi_{\vert h\cdot U_p}^{-1} \pi_{\vert U_p}$. Hence, $h$ acts holomorphically on $S\setminus Y$, since $\pi$ is holomorphic. Using this and the fact that $x\mapsto h\cdot x$ is continuous at $p\in Y$ we deduce from Riemann's theorem on removable singularities that the action on $S$ is holomorphic. We have proved the proposition.
\end{proof}

\begin{proposition}\label{prop ext}
Let $S$ be a closed surface, $Y\subset S$ a finite set and $f$ a self homeomorphism of $S\setminus Y$. There exists a unique self homeomorphism $\tilde{f}$ of $S$ extending $f$. Moreover, if $S$ is oriented and $f$ preserves the orientation, then $\tilde{f}$ is orientation preserving.
\end{proposition}
\begin{proof}
Denote by $y_1,\dots y_n$ the elements of $Y$ and take pairwise disjoint closed topological disks $D_i \subset S$ (for $n\in[1,n]$) such that $y_i$ lies in the interior of $D_i$. Denote by $\partial D_i$ the boundary of $D_i$ and set $D_i'=D_i\setminus \{y_i\}$. The punctured disc $f(D_i')$ is closed in $S\setminus Y$ but not compact. Hence, its closure in $S$ is obtained by adding points of $Y$. Take $y_{j_i}\in Y$ adherent to $f(D_i')$. Since $f(\partial D_i)$ is closed in $S$, small neighborhoods of $y_{j_i}$ do not intersect $f(\partial D_i)$ and such neighborhoods are included in $\{y_i\} \cup (f(D_i')\setminus f(\partial D_i))$ ($f(\partial D_i)$ disconnects the surface). In particular $y_{j_i}$, is not adherent to $f(D_k')$ for $k\neq i$. Hence, we have a bijection $y_i\mapsto y_{j_i}$ of $Y$ and the closure of $f(D_i')$ in $S$ is $f(D_i')\cup y_{j_i}$. Since $D_i$ and $f(D_i') \cup y_{j_i}$ are one point compactifications of $D_i'$ and $f(D_i')$ respectively, the homeomorphism $f_i:D_i'\to f(D_i')$ induced by $f$ extends to a homeomorphism $\tilde{f}_i: D_i\to f(D_i)'\cup y_{j_i}$ mapping $y_i$ to $y_{j_i}$. This gives an open and bijective extension $\tilde{f}: S\to S$ of $f$. Moreover, $\tilde{f}$ is unique, since $S\setminus Y$ is dense in $S$. The orientation part can be proved by manipulating the continuous local homological orientations on $S$ and $S\setminus Y$.
\end{proof}
\begin{corollary}\label{act extension}
Take $S,Y$ as in the proposition. Let $H$ be a group acting by homeomorphisms on $S\setminus Y$. The action of $H$ on $S\setminus Y$ extends to a unique action of $H$ on $S$. Moreover, if $S$ is oriented and the action on $S\setminus Y$ preserves the orientation, then the extended action preserves the orientation.
\end{corollary}

\begin{proposition}\label{prop Fresult}
Let $S$ be a closed, oriented surface, $Y$ a finite (eventually empty) subset of $S$ and $H$ be a finite group acting by orientation preserving homeomorphisms on $S\setminus Y$. The action of $H$ extends to a unique action of $H$ on $S$ stabilizing $Y$. Moreover, $S$ admits a complex structure in which the extended action of $H$ is holomorphic.
\end{proposition}
\begin{proof}
The proposition follows from proposition \ref{prop Com Str}, and corollary \ref{act extension}.
\end{proof}
  
\bibliographystyle{alpha}
\bibliography{Biblio}{}

\end{document}